\documentclass[11pt]{amsart}
\usepackage{amsmath,amsthm,amsfonts,amscd,amssymb,eucal,latexsym,mathrsfs}
\usepackage[all]{xy}

\newtheorem{theorem}{Theorem}[section]
\newtheorem{corollary}[theorem]{Corollary}
\newtheorem{lemma}[theorem]{Lemma}
\newtheorem{proposition}[theorem]{Proposition}

\theoremstyle{definition}
\newtheorem{definition}[theorem]{Definition}
\newtheorem{remark}[theorem]{Remark}

\newtheorem{example}[theorem]{Example}

\newtheorem{problem}[theorem]{Problem}

\newcommand{\spn}{{\rm span}}

\newcommand{\Fin}{{\mathscr F}}

\newcommand{\Ad}{{\rm Ad}\,}

\newcommand{\cb}{{\rm cb}}

\newcommand{\red}{{\lambda}}

\newcommand{\cB}{{\mathcal B}}

\newcommand{\cH}{{\mathcal H}}

\newcommand{\cU}{{\mathcal U}}

\newcommand{\cQ}{{\mathcal Q}}
\newcommand{\cP}{{\mathcal P}}

\newcommand{\Zb}{{\mathbb Z}}

\newcommand{\Tb}{{\mathbb T}}

\newcommand{\Rb}{{\mathbb R}}
\newcommand{\Nb}{{\mathbb N}}

\newcommand{\alg}{{\rm alg}}

\newcommand{\orb}{{\rm orb}}

\newcommand{\sB}{{\mathscr B}}
\newcommand{\sC}{{\mathscr C}}

\newcommand{\sH}{{\mathscr H}}

\newcommand{\sS}{{\mathscr S}}

\newcommand{\sP}{{\mathscr P}}

\newcommand{\OL}{{\rm{O}\mathcal{L}}}

\allowdisplaybreaks

\begin{document}

\title[Residually finite actions]{Residually finite actions and crossed products}

\author{David Kerr}
\author{Piotr W. Nowak}
\address{\hskip-\parindent
David Kerr, Department of Mathematics, Texas A{\&}M University,
College Station TX 77843-3368, U.S.A.}
\email{kerr@math.tamu.edu}

\address{\hskip-\parindent
Piotr W. Nowak, Department of Mathematics, Texas A{\&}M University,
College Station TX 77843-3368, U.S.A.}
\email{pnowak@math.tamu.edu}

\date{April 5, 2011}

\begin{abstract}
We study a notion of residual finiteness for continuous actions of discrete groups
on compact Hausdorff spaces and how it relates to the existence of norm microstates for the reduced crossed product.
Our main result asserts that an action of a free group on a zero-dimensional compact metrizable space
is residually finite if and only if its reduced crossed product admits norm microstates, i.e., is an MF algebra.
\end{abstract}

\maketitle

\section{Introduction}

Finite-dimensional approximation is a ubiquitous notion in the structure theory of $C^*$-algebras,
and it appears in a variety of different ways through properties like nuclearity, exactness,
quasidiagonality, and the existence of norm microstates (i.e., being an MF algebra) \cite{BroOza08}. Given that
reduced crossed products by actions of countable discrete groups on compact metrizable spaces
play an important role as examples and motivation 
in the study of $C^*$-algebras, one would like to understand the extent to which various forms of
finite-dimensional approximation in such crossed products are reflections of finite approximation 
properties at the level of the dynamics. 

Nuclearity and exactness are essentially measure-theoretic concepts and do not reflect anything 
intrinsically topological in the dynamics. Indeed nuclearity of the reduced crossed product is equivalent 
to the amenability of the action, which can be expressed in purely measure-dynamical terms \cite{AnaRen00}, 
while exactness of the reduced crossed product is equivalent to the exactness of the 
acting group \cite[Thm.\ 10.2.9]{BroOza08}, 
which can be characterized by the existence of an amenable action of the group on a compact 
metrizable space.

On the other hand, quasidiagonality and the existence of norm microstates
both involve the matricial approximation of multiplicative structure and are thus topological in nature,
and their precise relation to the dynamics is for the most part poorly understood. Actually much of the difficulty
stems from the group itself, as anything but the simplest geometry in the Cayley graph
can cause severe complications for an operator analysis based on perturbations, and obtaining a topological understanding
of the reduced crossed product amounts in part to knowning something about the reduced group $C^*$-algebra, which
sits inside the reduced crossed product in a canonical way. For free groups on two or more generators it is a deep
theorem of Haagerup and Thorbj{\o}rnsen that the reduced group $C^*$-algebra is MF \cite{HaaTho05}. For countable
discrete groups $G$, if the reduced group $C^*$-algebra $C^*_\red (G)$ is quasidiagonal then $G$ is amenable \cite{Had87},
while $C^*_\red (G)$ is residually finite-dimensional if and only if $G$ is residually finite and amenable \cite[Cor.\ 4]{Dad99}.
Not much else seems to be known however about quasidiagonality and the existence of norm microstates 
for reduced group $C^*$-algebras.

One important dynamical setting in which we do have a complete understanding of quasidiagonality is that of integer actions.
Pimsner showed in \cite{Pim83} that for a self-homeomorphism of a compact metrizable space
the following are equivalent: (i) the homeomorphism is pseudo-nonwandering (i.e., chain recurrent), (ii)
the crossed product is embeddable into an AF algebra, (iii) the crossed product is quasidiagonal.
Already though for $\Zb^2$-actions it is not clear what dynamical condition should correspond to quasidiagonality.
Lin succeeded however in showing that, for actions of $\Zb^d$ where $d\geq 1$, embeddability 
of the crossed product into a simple AF algebra is equivalent 
to the existence of an invariant Borel probability measure of full support \cite{Lin08}. 
It is also known that almost periodic actions of certain amenable groups produce
quasidiagonal crossed products \cite{Orf07}. Note that in all of these cases the reduced and full crossed products
coincide and are nuclear due to the amenability of the acting group, and that a nuclear $C^*$-algebra is
quasidiagonal if and only if it admits norm microstates, i.e., is an MF algebra.

The goal of the present paper is to initiate a dynamical study of the MF property in reduced crossed products
that is primarily targeted at actions of free groups. Such a program owes its very possibility
to the Haagerup-Thorbj{\o}rnsen result cited above which establishes the MF property of the reduced group $C^*$-algebras of
free groups on two or more generators, a conclusion whose validity is still unknown for the vast majority of nonamenable groups.
For actions of a countable discrete group $G$ on a compact metric space $X$ we define a dynamical version of residual finiteness 
for groups, which we again call residual finiteness (Definition~\ref{D-res finite}), 
and our main result asserts that, when $G$ is free and $X$ is zero-dimensional, 
the action is residually finite if and only if the reduced crossed product is an MF algebra (Theorem~\ref{T-mf}). 

For $G=\Zb$ residual finiteness is the same as chain recurrence and the situation reduces to a special case of the equivalence 
of (i) and (iii) in the statement of Pimsner's theorem above. Pimnser showed in \cite[Lemma 2]{Pim83} that, for a $\Zb$-action, 
chain recurrence is equivalent to every open set being topologically incompressible, and he establish the implication
(iii)$\Rightarrow$(i) by using compressibility to construct a nonunitary isometry in the crossed product as an obstruction 
to quasidiagonality. For a free group $F_r$ with $r\geq 2$ we do not have an analogue of this noncompressibility characterization 
for residual finiteness, and so our proof that an $F_r$-action on a zero-dimensional space is residually finite
if the reduced crossed product is MF must proceed by different means. The idea is to directly extract 
the finite approximations of the action from the existence of norm microstates.
To carry out the required perturbation arguments we need to be able to work with projections, which explains our
hypothesis that $X$ is zero-dimensional.

A residually finite action of $G$ on $X$ is, roughly speaking, one that approximately admits extensions consisting of actions of $G$ 
on finite sets, i.e., there are actions of $G$ on finite sets which map into $X$ approximately equivariantly with approximately dense image
(see Section~\ref{S-res finite}). This can be viewed as a topological analogue of soficity for actions preserving
a Borel probability measure, in which the approximations are measured in $2$-norm and the maps into $X$ must approximately
push forward the uniform measure to the given measure on $X$. One can count the exponential growth of the number of such 
approximately equivariant maps up to some observational error to define a notion of entropy 
(see \cite{Bow10a,KerLi10a} and especially Sections~2 and 3 of \cite{KerLi10b})
and this is one motivation for the study of finite dynamical approximations. 
It can be shown that all measure-preserving actions of $F_r$ on a 
standard probability space are sofic (see \cite{Bow10}), and that the von Neumann algebra crossed product of such actions admit 
tracial microstates, i.e., embed into the ultrapower $R^\omega$ of the hyperfinite II$_1$ factor (use \cite{BroDykJun10}).
So residual finiteness in our topological context can be thought of as playing the role of measure-preservingness,
while the MF property (existence of norm microstates) is the analogue of embeddability into $R^\omega$ (existence of tracial
microstates). Note however that in our setting one of the main points is to show that the MF property implies 
residual finiteness, while in the von Neumann algebra case there is nothing to prove in this direction.

After introducing residually finite actions in Section~\ref{S-res finite},
we show in the first part of Section~\ref{S-mf} that such actions give rise to an MF reduced crossed product 
whenever the reduced group $C^*$-algebra of the acting group is MF. 
In fact we prove this more generally in Theorem~\ref{T-qd mf} 
for actions on arbitrary separable $C^*$-algebras which are quasidiagonal in the sense of Definition~\ref{D-qd}. 
Using this fact we then proceed in the second part of Section~\ref{S-mf} to establish the
equivalence of residual finiteness and the MFness of the reduced crossed product in the case of a free group
acting on a zero-dimensional compact metrizable space (Theorem~\ref{T-mf}).
Motivated by recent work of R{\o}rdam and Sierakowski on paradoxical decompositions in the context of continuous actions 
on the Cantor set and purely infinite crossed products \cite{RorSie10}, we examine
in Section~\ref{S-pd} how paradoxical decomposability fits into our discussion of residual finiteness at the other extreme.
Section~\ref{S-minimal} gives a list of characterizations of residual finiteness 
for minimal actions of free groups which incorporates paradoxical decomposability among other phenomena.
In Section~\ref{S-affine} we use spaces of probability measures to construct,
for every nonamenable countable discrete group, an example of an action which is not residually finite
although its restriction to any cyclic subgroup is residually finite. Finally, in Section~\ref{S-integers}
we revisit integer actions and observe in this case, combining arguments and results from \cite{Had87,BlaKir07,Eck00},
that the following conditions are equivalent:
(i) the crossed product is a strong NF algebra, (ii) the $\OL_\infty$ invariant of the crossed product is $1$,
and (iii) there is collection of transitive residually finite subsystems with dense union.

We remark that Margulis and Vinberg defined in \cite{MarVin00} a nontopological notion of 
residual finiteness for actions in which the maps in the finite modelling go in the other direction.

We round out the introduction with a few words about notational convention.
Throughout the paper $G$ will always be a countable discrete group, with extra hypotheses
added explicitly whenever appropriate.
Actions of $G$ on a compact Hausdorff space $X$ will often be unnamed, in which case
simply write $G\curvearrowright X$, and they
will always be expressed using the concatenation $(s,x)\mapsto sx$ for $x\in X$ and $s\in G$. 
When necessary we will refer to actions by means of a symbol such as 
$\alpha$, in particular when we need to talk about the induced action on $C(X)$, which will 
actually be expressed using this symbol, i.e., $\alpha_s (f)(x) = f(s^{-1} x)$ for $f\in C(X)$, $x\in X$, and $s\in G$.
We write $\lambda$ for the left regular representation of $G$ on $\ell_2 (G)$.
The reduced crossed product of a continuous action $G\curvearrowright X$ will be written 
$C(X)\rtimes_\red G$ and the reduced group $C^*$-algebra of $G$ will be written $C^*_\red (G)$.
The canonical unitary in $C(X)\rtimes_\red G$ or $C^*_\red (G)$ associated to a group element $s$ will
invariably be denoted by $u_s$. For background on crossed products and group $C^*$-algebras, especially
from the kind of finite-dimensional approximation viewpoint of this paper, we refer the reader to \cite{BroOza08}.

For a compact Hausdorff space we write $M_X$ for the space of Borel probability measures on $X$
equipped with the weak$^*$ topology, under which it is a compact Hausdorff space.
Whenever convenient we will simultaneously regard elements of $M_X$ as states on $C(X)$.
The unitary group of a unital $C^*$-algebra $A$ will be written $\cU (A)$.
\medskip

\noindent{\it Acknowledgements.}
The first author was partially supported by NSF grant DMS-0900938 and the second author
partially supported by NSF grant DMS-0900874. We thank Adam Sierakowski for comments and corrections,
and Hanfeng Li for the argument in Example~\ref{E-alg} and the simple proof of
Lemma~\ref{L-invt measure res fin}.

\section{Residually finite actions}\label{S-res finite}

\begin{definition}\label{D-res finite}
A continuous action of $G$ on a compact Hausdorff space $X$ is said to be 
{\it residually finite} if for every finite set $F\subseteq G$ and neighbourhood $\varepsilon$
of the diagonal in $X\times X$ there are a finite set $E$, an action of $G$ on $E$,
and a map $\zeta : E\to X$ such that $\zeta (E)$ is $\varepsilon$-dense in $X$ and
$(\zeta (sz), s\zeta (z)) \in \varepsilon$ 
for all $z\in E$ and $s\in F$. 
\end{definition}

Note that if $S$ is a generating set for $G$ then to verify residual finiteness it is sufficient
for $F$ in the definition to be quantified over all finite subsets of $S$ by uniform continuity.

We will mostly be interested in actions on compact metrizable spaces, in which
case residual finiteness can be expressed in terms of a given compatible metric.
Indeed if $X$ is metrizable in the above definition and $d$ is a compatible metric on $X$ then 
residual finiteness is the same as saying that for every finite set $F\subseteq G$ and $\varepsilon > 0$
there are a finite set $E$, an action of $G$ on $E$,
and a map $\zeta : E\to X$ such that $\zeta (E)$ is $\varepsilon$-dense in $X$ and
$d(\zeta (sz), s\zeta (z)) < \varepsilon$ 
for all $z\in E$ and $s\in F$. 

Note that if $G$ admits a free residually finite continuous action
on a compact metrizable space, then $G$ must be a residually finite group, as is easy to see.
Indeed this is the motivation for our terminology in the dynamical setting.

In the case $G=\Zb$, residual finiteness is equivalent to chain recurrence, as explained at the beginning
of Section~\ref{S-integers}.

If $X$ has no isolated points then a simple perturbation argument shows that in Definition~\ref{D-res finite}
we can always take the set $E$ to be a subset of $X$ and the map $\zeta : E\to X$ to be the inclusion. However, 
for general $X$ this is not possible, as the following example demonstrates. 
For each $i=1,2,3$ take $X_i$ to be a copy of $\Zb$ and consider the two-point compactification $X_i \cup \{ y_i , z_i \}$ 
where $y_i$ and $z_i$ are the points at $-\infty$ and $+\infty$. Let $\Zb$ act on $X_i \cup \{ y_i , z_i \}$
so that it translates $X_i$ via addition and fixes $y_i$ and $z_i$. Now define $X$ by taking the quotient disjoint union 
of the sets $X_i \cup \{ y_i , z_i \}$ for $i=1,2,3$ which collapses each of the sets $\{ z_1 ,y_2 , y_3 \}$ 
and $\{ y_1 ,z_2 , z_3 \}$ to single points.

The following proposition shows that in the definition of residual finiteness it is enough to verify the 
$\varepsilon$-density condition in a local way.

\begin{proposition}\label{P-point}
A continuous action of $G$ on a compact Hausdorff space $X$ is residually finite if and only if for every $x\in X$, finite set 
$F\subseteq G$, and neighbourhood $\varepsilon$
of the diagonal in $X\times X$ there are a finite set $E$, an action of $G$ on $E$,
and a map $\zeta : E\to X$ such that $x$ lies in the $\varepsilon$-neighbourhood of $\zeta (E)$ and
$(\zeta (sz), s\zeta (z)) \in \varepsilon$ for all $z\in E$ and $s\in F$. 
\end{proposition}

\begin{proof} 
For the nontrivial direction, given a finite set $F\subseteq G$ and neighbourhood $\varepsilon$
of the diagonal in $X\times X$ take, for every $x$ in some finite $\varepsilon$-dense set $D\subseteq X$, a finite set $E_x$
and an action of $G$ on $E_x$, and a map $\zeta_x : E_x \to X$ such that $x$ lies in the $\varepsilon$-neighbourhood of 
$\zeta_x (E_x )$ and $(\zeta_x (sz), s\zeta_x (z)) \in \varepsilon$ for all $z\in E_x$ and $s\in F$. Now  
set $E = \coprod_{x\in D} E_x$ and define $\zeta : E \to X$ so that it restricts to $\zeta_x$ on $E_x$ for each $x\in D$.
Then $\zeta (E)$ is $\varepsilon$-dense in $X$ and
$(\zeta (sz), s\zeta (z)) \in \varepsilon$ for all $z\in E$ and $s\in F$, verifying residual finiteness. 
\end{proof}

\begin{proposition}\label{P-res fin measure}
Let $X$ be a compact Hausdorff space and $G\curvearrowright X$ a residually finite continuous action.
Then there is a $G$-invariant Borel probability measure on $X$.
\end{proposition}

\begin{proof}
Let $\Lambda$ be the net of pairs $(F,\varepsilon )$ where $F$ is a finite subset of $G$ and 
$\varepsilon$ is a neighbourhood of the diagonal in $X\times X$ and the order relation
$(F' ,\varepsilon' )\succ (F,\varepsilon )$ means that $F' \supseteq F$ and 
$\varepsilon' \subseteq\varepsilon$. 
For a given $(F,\varepsilon )\in\Lambda$ there exist, by residual finiteness, a finite set $D$, an action of $G$ on $E$,
and a map $\zeta : E\to X$ such that $(\zeta (sz), s\zeta (z)) \in \varepsilon$ 
for all $z\in E$ and $s\in F$, and we define $\mu_{F,\varepsilon}$
to be the pullback under $\zeta$ of the uniform probability measure on $E$, i.e., 
$\mu_{F,\varepsilon} (f) = |E|^{-1} \sum_{z\in E} f(\zeta (z))$ for $f\in C(X)$.
Now take a weak$^*$ limit point of the net $\{ \mu_{F,\varepsilon} \}_{(F,\varepsilon )\in\Lambda}$, 
which is easily verified to be a $G$-invariant Borel probability measure on $X$.
\end{proof}

\begin{example}
Suppose that $G$ is residually finite. Let $X$ be a compact Hausdorff space. Then the Bernoulli action
$G\curvearrowright X^G$ given by $s(x_t )_{t\in G} = (x_{s^{-1} t} )_{t\in G}$ 
is easily seen to be residually finite.  
\end{example}

\begin{example}\label{E-alg}
Let $f$ be an element in the group ring $\Zb G$. Then $G$ acts on $\Zb G/\Zb Gf$ by left translation, 
and this gives rise to an action $\alpha_f$ of $G$ by automorphisms on the compact Abelian dual group 
$X_f := \widehat{\Zb G/\Zb Gf}$. When $f$ is equal to $d$ times the unit for some $d\in\Nb$ we
obtain the Bernoulli action $G\curvearrowright \{ 1,\dots ,d \}^G$. 
Suppose now that $G$ is residually finite and $f$ is invertible as an element in the full group 
$C^*$-algebra $C^* (G)$. Then the action $\alpha_f$ is residually finite.
Indeed in this case the points in $X_f$ which are fixed
by some finite-index normal subgroup of $G$ are dense in $X_f$, which can be seen as follows.
Let $H$ be a finite-index normal subgroup
of $G$. Then the canonical surjective ring homomorphism $\pi_H : \Zb G \to \Zb (G/H)$
induces a surjective ring homomorphism $\Zb G / \Zb Gf \to \Zb (G/H) / \Zb (G/H)\pi_H (f)$ which in
turn induces an injective group homomorphism $X_{\pi_H (f)} \hookrightarrow X_f$. Note that 
$X_{\pi_H (f)}$, identified with its image in $X_f$, 
is equal to the set of points in $X_f$ which are fixed by $H$. It thus suffices to show
that $\bigcup_{H\in\sH} X_{\pi_H (f)}$ is dense in $X_f$ where $\sH$ denotes the collection of 
finite-index normal subgroups of $G$. This happens precisely when the natural homomorphism 
$\Zb G / \Zb Gf \to \prod_{H\in\sH} \Zb (G/H)/\Zb (G/H)\pi_H (f)$ is injective. 
To verify this injectivity, let $g$ be an element
in the kernel. Then for every $H\in\sH$ there is a $w_H \in \Zb (G/H)$ such that $\pi_H (g) = w_H \pi_H (f)$.
Since $f$ is invertible in $C^* (G)$, $w_H$ is unique and its $\ell_2$-norm is bounded above 
by some constant not depending on $H$. Now if $H_1 , H_2 \in\sH$ and $H_1 \subseteq H_2$
then the image of $w_{H_1}$ in $\Zb (G/H_2 )$ under the canonical map $\Zb (G/H_1 )\to\Zb (G/H_2 )$
is equal to $w_{H_2}$, and so we can define the projective limit of the $w_H$ as an 
element $w$ in $\Zb^G$. Then $g=wf$. Moreover, $w$ lies in $\ell_2 (G)$ because the $\ell_2$-norms of the $w_H$
are uniformly bounded, and so $w$ has finite support since it takes integer values. 
Therefore $g\in\Zb Gf$ and we obtain the desired injectivity.

We remark that when $f$ is invertible in $\ell_1 (G)$ 
and $\{ G_i \}_{i=1}^\infty$ is a sequence of finite-index normal subgroups of $G$ with 
$\bigcap_{j=1}^\infty \bigcup_{i=j}^\infty G_i = \{ e \}$, 
the measure entropy of $\alpha_f$ with respect to the Haar measure and the sofic approximation sequence $\Sigma$
arising from $\{ G_i \}_{i=1}^\infty$ is equal to the exponential growth rate 
of the number of $G_i$-fixed points and to the logarithm of the Fuglede-Kadison
determinant of $f$ in the group von Neumann algebra of $G$ \cite{Bow10a}. This is also true
for the topological entropy with respect to $\Sigma$ more generally whenever $f$ is invertible in the 
full group $C^*$-algebra \cite{KerLi10a}. 
\end{example}

\section{Residually finite actions and MF algebras}\label{S-mf}

A separable $C^*$-algebra is said to be an {\it MF algebra} if it can be expressed as the inductive limit
of a generalized inductive system of finite-dimensional $C^*$-algebras \cite[Defn.\ 3.2.1]{BlaKir97}.
This is equivalent to the embeddability of $A$ into $\prod_{n=1}^\infty M_{k_n} / \bigoplus_{n=1}^\infty M_{k_n}$
for some sequence $\{ k_n \}_{n=1}^\infty$ in $\Nb$ \cite[Thm.\ 3.2.2]{BlaKir97}, as well as to
the existence of norm microstates (see Section~11.1 of \cite{BroOza08}). The following characterizations
are tailored for our purposes.

\begin{proposition}\label{P-MF}
Let $A$ be a separable $C^*$-algebra.
The following are equivalent. 
\begin{enumerate}
\item $A$ is an MF algebra,

\item for every $\varepsilon > 0$ and finite
set $\Omega\subseteq A$ there are a $k\in\Nb$ and a $^*$-linear map $\varphi : A \to M_k$ such that
$\| \varphi (ab) - \varphi (a)\varphi (b) \| < \varepsilon$ and 
$\big| \| \varphi (a) \| - \| a \| \big| < \varepsilon$ for all $a,b\in\Omega$,

\item there is a dense $^*$-subalgebra $A_0$ of $A$ such that for every $\varepsilon > 0$ and finite
set $\Omega\subseteq A_0$ there are a $k\in\Nb$ and a $^*$-linear map $\varphi : A_0 \to M_k$ such that
$\| \varphi (ab) - \varphi (a)\varphi (b) \| < \varepsilon$ and 
$\big| \| \varphi (a) \| - \| a \| \big| < \varepsilon$ for all $a,b\in\Omega$.
\end{enumerate}
Moreover, if $A$ is unital then the maps $\varphi$ in (2) and (3) may be taken to be unital.
\end{proposition}

\begin{proof}
(1)$\Rightarrow$(2). 
By Proposition~2.2.3(iii) of \cite{BlaKir97}, $A$ embeds into $\prod_{n=1}^\infty M_{k_n} / \bigoplus_{n=1}^\infty M_{k_n}$
for some sequence $\{ k_n \}_{n=1}^\infty$ in such a way
that for each $a\in A$ one has $\lim_{n\to\infty} \| a_n \| = \| a \|$ for every $(a_n) \in\pi^{-1} (\{ a \})$
where $\pi : \prod_\lambda M_{k_n}\to \prod_\lambda M_{k_n} / \bigoplus_\lambda M_{k_n}$
is the quotient map.
Take a linear map $\psi : A\to\prod_{n=1}^\infty M_{k_n}$ such that $\pi\circ\psi = \Phi$. We may assume $\psi$
to be $^*$-linear by replacing it with 
$a\mapsto (\psi (a) + \psi (a^* )^* ))/2$ if necessary. For each $n$
let $\psi_n :A\to M_{k_n}$ be the composition of $\psi$ with the 
projection onto the $n$th factor. Then we have
$\lim_{n\to\infty} \| \psi_n (a) \| = \| a \|$ and 
$\lim_{n\to\infty} \| \psi_n (ab) - \psi_n (a)\psi_n (b) \| = 0$
for all $a,b\in A$, from which (2) follows. 

(2)$\Rightarrow$(3). Trivial.

(3)$\Rightarrow$(1). Since $A$ is separable there is an increasing sequence 
$\Omega_1 \subseteq\Omega_2 \subseteq\cdots$ of finite subsets of $A_0$ such that $\bigcup_{n=1}^\infty \Omega_n^n$
is a dense $^*$-subalgebra $A_1$ of $A$. For each $n\in\Nb$ take a $k_n \in\Nb$ and a $^*$-linear map 
$\varphi_n : A_0 \to M_{k_n}$ such that $\| \varphi_n (ab) - \varphi_n (a)\varphi_n (b) \| < \varepsilon$ and 
$\big| \| \varphi_n (a) \| - \| a \| \big| < \varepsilon$ for all $a,b\in\Omega_n^n$. Now define a map
$\theta : A \to\prod_{n=1}^\infty M_{k_n} \big/ \bigoplus_{n=1}^\infty M_{k_n}$
by composing $a\mapsto (\varphi_n (a))_{n=1}^\infty$ with the canonical projection map
$\prod_{n=1}^\infty M_{k_n} \to \prod_{n=1}^\infty M_{k_n} \big/ \bigoplus_{n=1}^\infty M_{k_n}$.
Then $\theta$ is isometric and multiplicative, and hence extends to an injective $^*$-homomorphism
$A\to \prod_{n=1}^\infty M_{k_n} \big/ \bigoplus_{k=1}^\infty M_{k_n}$, yielding (1).

In the case that $A$ is unital, note that in the proof of (1)$\Rightarrow$(2) the embedding $\Phi$ 
may be taken to be unital, since we can lift the image of the unit under 
$\Phi$ to a projection $(p_n )$ in $\prod_\lambda M_{k_n}$, 
producing an injective unital $^*$-homomorphism from $A$ to 
$\prod_{n=1}^\infty p_n M_{k_n} p_n / \bigoplus_{n=1}^\infty p_n M_{k_n} p_n$. This enables us
to arrange $\psi$ to be unital, so that each $\psi_n$ is unital. It follows that 
the maps $\varphi$ in conditions (2) and (3) may be taken to be unital.
\end{proof}

Note that the maps $\varphi$ in Proposition~\ref{P-MF} are typically not bounded.

A $C^*$-algebra is said to be {\it quasidiagonal} if it admits a faithful representation whose image is a
quasidiagonal set of operators. Voiculescu showed that a separable $C^*$-algebra $A$ is quasidiagonal if and only if 
for every $\varepsilon > 0$ and finite
set $\Omega\subseteq A$ there are a $n\in\Nb$ and a contractive completely positive map $\varphi : A \to M_n$ such that
$\| \varphi (ab) - \varphi (a)\varphi (b) \| < \varepsilon$ and 
$\| \varphi (a) \| \geq \| a \| - \varepsilon$ for all $a,b\in\Omega$ \cite{Voi91}.
Blackadar and Kirchberg showed that a separable nuclear $C^*$-algebra is an MF algebra if and only 
if it is quasidiagonal \cite[Thm.\ 5.2.2]{BlaKir97}. The reduced group $C^*$-algebra of a free group
on two or more generators is an example of an MF algebra which is not quasidiagonal \cite{HaaTho05}.

Voiculescu's abstract characterization of quasdiagonality motivates the following definition,
which extends the concept of residual finiteness to actions on noncommutative $C^*$-algebras 
in view of Proposition~\ref{P-rf qd}. We will formulate and prove Theorems~\ref{T-qd mf} and \ref{T-qd qd} 
within this general noncommutative framework.

\begin{definition}\label{D-qd}
Let $\alpha$ be an action of $G$ on a separable $C^*$-algebra $A$. 
We say that $\alpha$ is {\it quasidiagonal} 
if for every finite set $\Omega\subseteq A$, finite set $F\subseteq G$, and $\varepsilon > 0$ there
are an $d\in\Nb$, an action $\gamma$ of $G$ on $M_d$, and a unital completely positive map $\varphi : A\to M_d$ such that
\begin{enumerate}
\item $\| \varphi (ab) - \varphi (a)\varphi (b) \| < \varepsilon$ for all $a,b\in\Omega$,

\item $\| \varphi (a) \| \geq \| a \| - \varepsilon$ for all $a\in\Omega$, and

\item $\| \varphi (\alpha_s (a)) - \gamma_s (\varphi (a)) \| < \varepsilon$ for all $a\in\Omega$ and $s\in F$.
\end{enumerate} 
\end{definition}

Note that the existence of a quasidiagonal action on $A$ in the sense of the above definition implies that 
$A$ is quasidiagonal as a $C^*$-algebra.

\begin{proposition}\label{P-rf qd}
Let $\alpha$ be a residually finite continuous action of $G$ 
on a compact metrizable space $X$. Then $\alpha$ is quasidiagonal as an action of $G$ on $C(X)$.
\end{proposition}

\begin{proof}
Fix a compatible metric $d$ on $X$. Let $\varepsilon > 0$ and let $\Omega$ be a finite subset of $C(X)$.
By uniform continuity there is a $\delta > 0$ such that $|f(x) - f(y)| < \varepsilon$ for all $f\in\Omega$ 
and all points $x,y\in X$ satisfying $d(x,y) < \delta$.
By residual finiteness there are a finite set $E$, an action $\gamma$ of $G$ on $E$,
and a map $\zeta : E\to X$ such that $\zeta (E)$ is $\delta$-dense in $X$ and
$d(\zeta (s^{-1} z), s^{-1} \zeta (z)) < \delta$ for all $z\in E$ and $s\in F$. 
Define $\varphi : C(X)\to C(E)$ by $\varphi (f)(z) = f(\zeta (z))$ for all $f\in C(X)$ and $z\in E$.
Then $\varphi$ is a homomorphism, and for all $f\in\Omega$ and $s\in F$ we have
\[
\| \varphi (f) \| = \sup_{x\in\zeta (E)} |f(x)| \geq \| f \| - \varepsilon 
\]
and
\[
\| \varphi (\alpha_s (f)) - \gamma_s (\varphi (f)) \| 
= \sup_{z\in E} \big| f(s^{-1} \zeta (z)) - f(\zeta (s^{-1} z) \big| < \varepsilon ,
\]
showing that $\alpha$ is quasidiagonal.
\end{proof}

The proof of the following theorem is reminiscent of part of the proof of Theorem~7 in \cite{Pim83}.

\begin{theorem}\label{T-qd mf}
Suppose that $C^*_\red (G)$ is an MF algebra.
Let $\alpha$ be a quasidiagonal action of $G$ on a separable $C^*$-algebra $A$.
Then $A\rtimes_{\red} G$ is an MF algebra.
\end{theorem}

\begin{proof}
We view $A$ as acting on a separable Hilbert space $\cH$ via some faithful essential representation and
$A\rtimes_{\red} G$ as acting on $\cH\otimes\ell^2 (G)$ in the standard way, as determined
by $au_s (\xi\otimes\delta_t ) = \alpha_{st}^{-1} (a)\xi\otimes\delta_{st}$ 
for $a\in A$, $s,t\in G$, and $\xi\in\cH$, where $\{ \delta_s : s\in G \}$ is the canonical basis of $\ell^2 (G)$.

Since $\alpha$ is quasidiagonal there exist, for each $n\in\Nb$, a positive integer $k_n$,
an action $\gamma_n$ of $G$ on $M_{k_n}$, and a unital completely positive map $\varphi_n : A\to M_{k_n}$
so that
\begin{enumerate}
\item $\lim_{n\to\infty} \| \varphi_n (ab) - \varphi_n (a) \varphi_n (b) \| = 0$ for all $a,b\in A$,

\item $\lim_{n\to\infty} \| \varphi_n (a) \| = \| a \|$ for all $a\in A$, and

\item $\lim_{n\to\infty} \| \varphi_n (\alpha_s (a)) - \gamma_{n,s} (\varphi_n (a)) \| = 0$ for all $a\in A$ and $s\in G$.
\end{enumerate}
Note that if we view $M_{k_n}$ as acting on $L^2 (M_{k_n} ,\tau )$ via the GNS construction with respect to
the unique tracial state $\tau$, then the action $\gamma$ is implemented through conjugation by the
unitaries $w_s$ for $s\in G$ defined by $w_s x\eta = \gamma_s (x)\eta$ where $\eta$ is the canonical cyclic vector. 
Thus by replacing $M_{k_n}$ with $\cB (L^2 (M_{k_n} ,\tau ))\cong M_{k_n^2}$ and relabeling we may assume that
there are unitary representations $w_n : G\to\cU (\ell_2^{k_n} )$ such that
\[
\lim_{n\to\infty} \| \varphi_n (\alpha_s (a)) - w_{n,s} \varphi_n (a) w_{n,s}^* \| = 0
\]
for all $a\in A$ and $s\in G$.

Let $\Omega$ be a finite subset of the algebraic crossed product $A\rtimes_\alg G$, which by definition consists of
sums of the form $\sum_{s\in F} a_s u_s$ where $F$ is a finite subset of $G$ and each $a_s$ lies in $A$. Let $\varepsilon > 0$.
We will show the existence of a $d\in\Nb$ and $^*$-linear map $\beta : A\rtimes_\alg G \to M_d$ such that 
$\| \beta (cc') - \beta (c)\beta (c' ) \| < \varepsilon$ and $\big| \| \beta (c) \| - \| c \| \big| < \varepsilon$
for all $c,c'\in\Omega$. This is sufficient to complete the proof by Proposition~\ref{P-MF}.

Regard $M_{k_n}$ as acting on $\ell_2^{k_n}$ in the standard way.
For $N\in\Nb$ we define the unital completely positive map 
$\Phi_N : A\to \bigoplus_{n=N}^\infty M_{k_n}$
by 
\[
\Phi_N (a) = (\varphi_N (a), \varphi_{N+1} (a), \dots ) .
\] 
Write $D_n$ for the $C^*$-subalgebra of $\cB (\ell_2^{k_n} \otimes\ell_2 (G))$
generated by $M_{k_n} \otimes 1$ and the operators $w_{n,s} \otimes\lambda_s$ for $s\in G$.
Define a $^*$-linear map 
$\Theta_N : A\rtimes_\alg G \to \bigoplus_{n=N}^\infty D_n$ by setting
\[
\Theta_N (au_s ) = (\varphi_N (a)w_{N,s} \otimes\lambda_s , \varphi_{N+1} (a)w_{N+1,s} \otimes\lambda_s ,\dots )
\]
for $a\in A$ and $s\in G$ and extending linearly, which we can do 
since the subspaces $Au_s$ for $s\in G$ are orthogonal with respect to the canonical conditional
expectation from $A\rtimes_\red G$ onto $A$.
We aim to show the existence of an $N\in\Nb$ such that
\begin{enumerate}
\item $\| \Theta_N (cc') - \Theta_N (c)\Theta_N (c') \| < \varepsilon /2$ for all $c,c' \in \Omega$, and

\item $\big| \| \Theta_N (c) \| - \| c \| \big| < \varepsilon /2$ for all $c \in \Omega$.
\end{enumerate}

Since $\Omega$ is finite we can find an $\varepsilon' > 0$ such that, for all $n\in\Nb$ and $c\in\Omega$, if 
$\big| \| \Theta_n (c) \|^2 - \| c \|^2 \big| < \varepsilon'$ then $\big| \| \Theta_n (c) \| - \| c \| \big| < \varepsilon /2$.
Take a finite set $F\subseteq G$ such that for every $c\in\Omega$ we can write
$c = \sum_{s\in F} a_{c,s} u_s$ where $a_{c,s} \in A$ for each $s\in F$. 
Set $M = \max_{c\in\Omega ,s\in F} \| a_{c,s} \|$.
For each $c\in\Omega$ take a unit vector $\eta_c$ in $\cH\otimes\ell^2 (G)$ such that
$\| c \| \leq \| c \eta_c \| + \varepsilon' /2$
and $\eta_c = \sum_{t\in K} \xi_{c,t} \otimes\delta_t$ for some finite set $K\subseteq G$ and vectors $\xi_{c,t} \in\cH$.
For notational simplicity we may assume $K$ to be the same for all $c\in\Omega$.

Take a $\delta > 0$ such that $(1+2|F||K|)|F|^2 \delta^2 \leq \varepsilon' /2$. 
Since the unital completely positive maps $\varphi_n$ 
are asymptotically multiplicative and asymptotically isometric, 
by the version of Voiculescu's theorem which appears as Theorem~2.10 in \cite{Bro04} and stems from \cite{Sal77}
we can find a $N\in\Nb$ and a unitary $U : \bigoplus_{n=N}^\infty \ell_2^{k_n} \to \cH$ 
such that
\begin{align*}
\big\| U\Phi_N (\alpha_{st}^{-1} (a_{c,s} ))  U^{-1} - \alpha_{st}^{-1} (a_{c,s} )\big\| < \frac{\delta}{2} 
\end{align*}
for all $c\in\Omega$, $s\in F$, and $t\in K$.
For $s\in G$ set $w_s = (w_{N,s} , w_{N+1,s} ,\dots ) \in \bigoplus_{n=N}^\infty M_{k_n}$.
By the asymptotic equivariance of the maps $\varphi_k$, we may assume that $N$ is large enough so that
\begin{align*}
\big\| \Phi_N (\alpha_{st}^{-1} (a_{c,s} )) - w_{st}^* \Phi_N (a_{c,s} )w_{st} \big\| < \frac{\delta}{2}
\end{align*}
and hence
\begin{align*}
\big\| w_{st} U^{-1} \alpha_{st}^{-1} (a_{c,s} ) - \Phi_N (a_{c,s} )w_{st} U^{-1} \big\| < \delta
\end{align*}
for all $c\in\Omega$, $s\in F$, and $t\in K$. We may moreover assume that $N$ is large enough so that
$\| \Theta_N (cc') - \Theta_N (c)\Theta_N (c') \| < \varepsilon /2$ for all $c,c' \in \Omega$, since
for $a,b\in A$, and $s,t\in G$ we have
\begin{align*}
\| \Theta_N (au_s bu_t ) - \Theta_N (au_s )\Theta_N (bu_t ) \| 
&= \sup_{n\geq N} \| \varphi_n (a\alpha_s (b))w_{n,st} - \varphi_n (a)w_{n,s} \varphi_n (b)w_{n,t} \| \\
&\leq \sup_{n\geq N} \| (\varphi_n (a\alpha_s (b)) - \varphi_n (a) \varphi_n (\alpha_s (b)))w_{n,st} \| \\
&\hspace*{10mm} \ + \sup_{n\geq N} \| \varphi_n (a) (\varphi (\alpha_s (b)) w_{n,s} - w_{n,s} \varphi_n (b)) w_{n,t} \| ,
\end{align*}
where the last two suprema tend to zero as $N\to\infty$, and the subspaces $Au_s$ for $s\in G$ 
linearly span $A\rtimes_\alg G$.

Let $c\in\Omega$. Write $\tilde{U}$ for the unitary operator from 
$( \bigoplus_{n=N}^\infty \ell_2^{k_n} ) \otimes \ell_2 (G)$ to $\cH\otimes \ell_2 (G)$
given by $\tilde{U} (\zeta\otimes\delta_t ) = Uw_t^{-1} \zeta\otimes\delta_t$.
Since $\eta_c$ is a unit vector, for $s\in F$ we have
\begin{align*}
\lefteqn{\bigg\| \sum_{t\in K} \big( w_{st} U^{-1} \alpha_{st}^{-1} (a_{c,s} ) 
- \Phi_N (a_{c,s} )w_{st} U^{-1} \big)\xi_{c,t} \otimes\delta_{st} \bigg\|^2}\hspace*{30mm} \\
\hspace*{30mm} &= \sum_{t\in K} \big\| \big( w_{st} U^{-1} \alpha_{st}^{-1} (a_{c,s} )
- \Phi_N (a_{c,s} )w_{st} U^{-1} \big)\xi_{c,t} \big\|^2 \\
&\leq \sum_{t\in K} \big\| w_{st} U^{-1} \alpha_{st}^{-1} (a_{c,s} )
- \Phi_N (a_{c,s} )w_{st} U^{-1} \big\|^2 \| \xi_{c,t} \|^2 \\
&< \delta^2 .
\end{align*}
For any vectors $x_1 , \dots ,x_n , y_1 , \dots ,y_n \in\cH$ we have 
\begin{align*}
\bigg\|\sum_{i=1}^n x_i \bigg\|^2 &\leq 
\bigg( \bigg\|\sum_{i=1}^n y_i \bigg\| + \bigg\|\sum_{i=1}^n (x_i - y_i )\bigg\| \bigg)^2 \\
&\leq \bigg\|\sum_{i=1}^n y_i \bigg\|^2 
+ \bigg(1+2\bigg\|\sum_{i=1}^n y_i \bigg\|\bigg)\bigg\|\sum_{i=1}^n (x_i - y_i )\bigg\|^2 
\end{align*} 
and so applying this inequality and the crude bound
\begin{align*}\tag{$\ast$}
\| \Theta_N (c) \tilde{U}^{-1} \eta_c \|
= \bigg\| \sum_{s\in F} \sum_{t\in K} \Phi_N (a_{c,s} )w_{st} U^{-1} \xi_{c,t} \otimes\delta_{st} \bigg\| 
\leq \sum_{s\in F} \sum_{t\in K} \| \Phi_N (a_{c,s} ) \| 
\leq |F||K| .
\end{align*}
we obtain
\begin{align*}
\| c\eta_c \|^2 
&= \bigg\| \tilde{U}^{-1} \sum_{s\in F} \sum_{t\in K} \alpha_{st}^{-1} (a_{c,s} ) \xi_{c,t} \otimes\delta_{st} \bigg\|^2 \\ 
&= \bigg\| \sum_{s\in F} \sum_{t\in K} w_{st} U^{-1} \alpha_{st}^{-1} (a_{c,s} ) \xi_{c,t} \otimes\delta_{st} \bigg\|^2 \\ 
&\leq \| \Theta_N (c) \tilde{U}^{-1} \eta_c \|^2 
+ \big(1+2\| \Theta_N (c) \tilde{U}^{-1} \eta_c \| \big) \\
&\hspace*{25mm} \ \times \bigg( \sum_{s\in F} \bigg\| \sum_{t\in K} 
\big( w_{st} U^{-1} \alpha_{st}^{-1} (a_{c,s} ) - \Phi_N (a_{c,s} )w_{st} U^{-1} \big)\xi_{c,t} \otimes\delta_{st} \bigg\|
\bigg)^2 \\
&\leq \| \Theta_N (c) \|^2 + (1+2|F||K|)|F|^2 \delta^2 \\
&= \| \Theta_N (c) \|^2 + \frac{\varepsilon'}{2} .
\end{align*}
Consequently
$\big\| \Theta_N (c) \big\|^2 \geq \| c \eta_c \|^2 - \varepsilon' /2 \geq \| c \|^2 - \varepsilon'$.

Next let us show that 
$\| c \|^2 \geq \| \Theta_N (c) \|^2 - \varepsilon'$. 
Take a unit vector $\eta$ in $\cH\otimes\ell^2 (G)$ such that
$\| \Theta_N (c) \| \leq \| \Theta_N (c) \tilde{U}^{-1} \eta \| + \varepsilon' /2$. The idea is to argue as in the above paragraph,
reversing the roles of $c$ and $\Theta_N (c)$ and replacing $\eta_c$ with $\eta$. Notice that 
the only way the particular choice of the vectors $\eta_c$ entered into the above estimates, besides their
being of norm one, was in obtaining the bound ($\ast$). Here however we can simply take 
$\| \tilde{U}^{-1} c \eta \| \leq \| c \|$ 
as the replacement for this bound, in which case
\begin{align*} 
\| \Theta_N (c) \tilde{U}^{-1} \eta \|^2
&= \bigg\| \sum_{s\in F} \sum_{t\in K} \Phi_N (a_{c,s} ) w_{st} U^{-1} \xi_{c,t} \otimes\delta_{st} \bigg\|^2 \\ 
&\leq \| \tilde{U}^{-1} c\eta \|^2 + \big(1+2\| \tilde{U}^{-1} c\eta \| \big) \\
&\hspace*{15mm} \ \times
\bigg( \sum_{s\in F} \bigg\| \sum_{t\in K} \big( \Phi_N (a_{c,s} )w_{st} U^{-1} 
- w_{st} U^{-1} \alpha_{st}^{-1} (a_{c,s} ) \big)\xi_{c,t} \otimes\delta_{st} \bigg\| \bigg)^2 \\
&\leq \| c \|^2 + 3|F|^2\delta^2 \\
&= \| c \|^2 + \frac{\varepsilon'}{2} 
\end{align*}
and hence $\| c \|^2 \geq \| \Theta_N (c) \tilde{U}^{-1} \eta \|^2 - \varepsilon' /2 \geq \| \Theta_N (c) \|^2 - \varepsilon'$.
By our choice of $\varepsilon'$ we conclude that $\big| \| \Theta_N (c) \| - \| c \| \big| < \varepsilon /2$ for all $c\in\Omega$, 
so that the map $\Theta_N$ satisfies the desired properties.

Now let $M$ be an integer greater than or equal to $N$ to be determined shortly.
Set $D = \bigoplus_{n=N}^M D_n$.
Define the $^*$-linear map $\theta : A\rtimes_\alg G \to D$ by declaring that
\[
\theta (au_s ) = (\varphi_N (a)w_{n,s} \otimes\lambda_s , \dots , \varphi_M (a)w_{n,s} \otimes\lambda_s ) 
\] 
for $a\in A$ and $s\in G$ and extending linearly.
In view of the properties of $\Theta_N$ we have $\| \theta (cc') - \theta (c)\theta (c') \| < \varepsilon /3$ 
for all $c,c' \in \Omega$ and, by taking $M$ large enough, 
$\big| \| \theta (c) \| - \| c \| \big| < \varepsilon /2$ for all $c \in \Omega$.

Set $k = \sum_{n=N}^M k_n$. Note that $D$ is a $C^*$-subalgebra of
$\cB (\bigoplus_{n=N}^M \ell_2^{k_n} )\otimes C^*_\red (G)$ as canonically represented on 
$(\bigoplus_{n=N}^M \ell_2^{k_n} )\otimes\ell_2 (G)$ where the latter is identified in the standard way 
with $\bigoplus_{n=N}^M (\ell_2^{k_n} \otimes\ell_2 (G))$. 
Via some fixed identification of $\cB (\bigoplus_{n=N}^M \ell_2^{k_n} )$ with $M_k$ we view
$D$ as a $C^*$-subalgebra of $M_k \otimes C^*_\red (G)$.
By hypothesis $C^*_\red (G)$ is MF, and hence so is $M_k \otimes C^*_\red (G)$, for
if $C^*_\red (G) \hookrightarrow \prod_{n=1}^\infty M_{k_n} \big/ \bigoplus_{n=1}^\infty M_{k_n}$
is an embedding witnessing the fact that $C^*_\red (G)$ is MF then we obtain an embedding
\[
M_k \otimes C^*_\red (G) \hookrightarrow 
M_k \otimes \Bigg(\prod_{n=1}^\infty M_{k_n} \bigg/ \bigoplus_{n=1}^\infty M_{k_n} \Bigg) \cong 
\prod_{n=1}^\infty \big( M_k \otimes M_{k_n} \big) \bigg/ \bigoplus_{n=1}^\infty \big( M_k \otimes M_{k_n} \big) .
\]
Thus by Proposition~\ref{P-MF} there is a $d\in\Nb$ and a $^*$-linear map 
$\varphi : M_k \otimes C^*_\red (G)\to M_d$ such that,
for all $c,c' \in\Omega$,
\begin{enumerate}
\item $\| \varphi (\theta (c) \theta (c')) - \varphi (\theta (c)) \varphi (\theta (c')) \| < \varepsilon /3$,

\item $\| \varphi (\theta (cc') - \theta (c) \theta (c')) \| 
< \| \theta (cc') - \theta (c) \theta (c') \| + \varepsilon /3$, and

\item $| \| \varphi (\theta (c)) \| - \| \theta (c) \| \big| < \varepsilon /2$.
\end{enumerate}
Set $\beta = \varphi\circ\theta$. Then $\beta$ is $^*$-linear, and for all $c,c'\in\Omega$ we have
\begin{align*}
\| \beta (cc') - \beta (c)\beta (c' ) \|
&\leq \| \varphi (\theta (cc' ) - \theta (c) \theta (c' )) \| +
\| \varphi (\theta (c) \theta (c' )) - \varphi (\theta (c)) \varphi (\theta (c' )) \| \\
&\leq \| \theta (cc' ) - \theta (c) \theta (c' ) \| + \frac{\varepsilon}{3} + \frac{\varepsilon}{3} 
< \varepsilon 
\end{align*}
and
\begin{align*}
\big| \| \beta (c) \| - \| c \| \big| 
\leq \big| \| \varphi (\theta (c)) \| - \| \theta (c) \| \big| 
+ \big| \| \theta (c) \| - \| c \| \big|
< \frac{\varepsilon}{2} + \frac{\varepsilon}{2} = \varepsilon ,
\end{align*}
completing the proof.
\end{proof}

\begin{theorem}\label{T-qd qd}
Suppose that $C^*_\red (G)$ is quasidiagonal.
Let $\alpha$ be a quasidiagonal action of $G$ on a separable nuclear $C^*$-algebra $A$.
Then $A\rtimes_{\red} G$ is quasidiagonal.
\end{theorem}

\begin{proof}
By a result of Rosenberg (see the appendix of \cite{Had87}), the quasidiagonality of $C^*_\red (G)$ implies that $G$ is amenable. 
Since $A$ is nuclear, it follows that the crossed product $A\rtimes_{\red} G$ is nuclear (see Section~IV.3.5 of \cite{Bla06}). 
Since $C^*_\red (G)$ is quasidiagonal it is an MF algebra, and so $A\rtimes_{\red} G$ is an MF algebra by Theorem~\ref{T-qd mf}.
Since separable nuclear MF algebras are quasidiagonal \cite[Thm.\ 5.2.2]{BlaKir97}, we conclude that $A\rtimes_{\red} G$ 
is quasidiagonal.
\end{proof}

An action of $G$ on a compact Haudorff space $X$ is said to be {\it topologically free} if the set
of points in $X$ with trivial isotropy group is dense.

\begin{corollary}
Suppose that $G$ is amenable. Let $G\curvearrowright X$ be a topologically free residually finite action 
on a compact metrizable space. Then $C(X) \rtimes_{\red} G$ is quasidiagonal.
\end{corollary}

\begin{proof}
Because it admits a topologically free residually finite action, $G$ must be a residually finite group, as
is easy to verify. Since the full and reduced group $C^*$-algebras of $G$ coincide by amenability, 
it follows that $C_\red^* (G)$ is quasidiagonal \cite[Cor.\ 4]{Dad99}. 
Since $C(X)$ is nuclear we obtain the conclusion by Theorem~\ref{T-qd qd}.
\end{proof}

Now we fix an $r\in \{ 2,3,\dots,\infty \}$ and concentrate on actions of the free group $F_r$
for the remainder of the section. Our aim is to show that, for continuous actions of $F_r$
on a zero-dimensional compact metrizable space, residual finiteness is equivalent to the reduced crossed product
being an MF algebra. Note that such a crossed product cannot be quasidiagonal since it contains $C^*_\red (F_r )$,
which, by a result of Rosenberg \cite{Had87}, is not quasidiagonal since $F_r$ is nonamenable.
The property of being an MF algebra is the appropriate substitute for quasidiagonality in the nonamenable case.
Indeed quasidiagonality and the property of being an MF algebra are equivalent for separable nuclear $C^*$-algebras
\cite[Thm.\ 5.2.2]{BlaKir97}, and so quasidiagonality in the general group action setting can viewed as an artifact of nuclearity.

The following two lemmas are standard types of perturbation results.

\begin{lemma}\label{L-proj cut}
Let $\eta > 0$. Then there is a $\delta > 0$ such that whenever
$d\in\Nb$ and $p$ and $q$ are projections in $M_d$ satisfying $\| pq \| < \delta$ there exists
a projection $q' \in (1-p)M_d (1-p)$ such that $\| q' - q \| < \eta$.
\end{lemma}

\begin{proof}
Let $\delta$ be a strictly positive number less than $\eta /6$ to be further specified.
Let $p$ and $q$ be projections in some matrix algebra $M_d$ such that $\| pq \| < \delta$.
Set $a = (1-p)q(1-p)$. Then 
\begin{align*}
\| q - a \| = \| pq+qp-pqp \| &\leq 3 \| pq \| < 3\delta .
\end{align*}
Since 
\begin{align*}
\| a^2 -a \| \leq \| a^2 - q^2 \| + \| q - a \| < \| (a-q)a \| + \| q(a-q) \| + 3\delta < 9\delta ,
\end{align*}
if $\delta$ is small enough as a function of $\eta$ there exists, by the functional calculus, a projection $q' \in (1-p)M_d (1-p)$
such that $\| q' - a \| < \eta /2$. Since $\delta < \eta /6$, we have
$\| q' - q \| \leq \| q' - a \| + \| a - q \| < \eta /2 + \eta /2 = \eta$, as desired.
\end{proof}

\begin{lemma}\label{L-hom}
Let $n\in\Nb$ and $\varepsilon > 0$. Then there exists a $\delta > 0$ such that whenever $d\in\Nb$ and 
$a_1 ,\dots ,a_n$ are $n$ self-adjoint elements in $M_d$ satisfying
\begin{enumerate}
\item $\| a_i^2 - a_i \| < \delta$ for all $i=1,\dots ,n$,

\item $\| a_i a_j \| < \delta$ for all distinct $i,j=1,\dots ,n$, and

\item $\sum_{i=1}^n a_i = 1$
\end{enumerate}
there exist pairwise orthogonal projections $p_1 ,\dots ,p_n \in M_d$ such that $\sum_{i=1}^n p_i = 1$ 
and $\| p_i - a_i \| < \varepsilon$ for all $i=1,\dots ,n$.
\end{lemma}

\begin{proof}
We may assume that $\varepsilon \leq 1$. 
Using Lemma~\ref{L-proj cut} we successively choose numbers $\delta_n > \delta_{n-1} > \dots > \delta_1 > 0$ such that
$\delta_n = \varepsilon /2n$ and, for each $k=1,\dots ,n-1$, whenever
$d\in\Nb$ and $p$ and $q$ are projections in $M_d$ satisfying $\| pq \| < (k+1)\delta_k$ there exists
a projection $q' \in (1-p)M_d (1-p)$ such that $\| q' - q \| < \delta_{k+1}$.
Let $\delta$ be a strictly positive number less than $\delta_1 /4n$ to be further specified.
Let $a_1 ,\dots ,a_n$ be $n$ self-adjoint elements in some matrix algebra $M_d$ such that
$\| a_i^2 - a_i \| < \delta$ for all $i=1,\dots ,n$, $\| a_i a_j \| < \delta$ for all distinct $i,j=1,\dots ,n$, and
$\sum_{i=1}^n a_i = 1$. By the functional calculus, we may take $\delta$ to be small enough as a function of 
$\varepsilon$ in order to be able to find projections $q_1 , \dots , q_n \in M_d$ such that 
$\| a_i - q_i \| < \varepsilon /2n$ for all $i=1,\dots ,n$.
Then for all distinct $i,j=1,\dots ,n$ we have 
\begin{align*}
\| q_i q_j \| &\leq \| (q_i - a_i )q_j \| + \| a_i (q_j - a_j ) \| + \| a_i a_j \| \\
&\leq \| q_i - a_i \| + (\| q_i \| + \| a_i - q_i \| ) \| q_j - a_j \| + \delta < 4\delta .
\end{align*}

Using our choice of the numbers $\delta_n ,\dots ,\delta_1$ as given by Lemma~\ref{L-proj cut},
we successively construct pairwise orthogonal projections
$p_1 , \dots ,p_n \in M_d$ such that $\| p_i - q_i \| < \varepsilon /2n$ for all $i=1,\dots ,n$.
More precisely, at the $k$th stage we use the fact that the projections $p = \sum_{i=1}^k p_i$ and $q = q_{k+1}$ satisfy
\begin{align*}
\| pq \| \leq \sum_{i=1}^k \| p_i q_{k+1} \| &\leq \sum_{i=1}^k \big( \| p_i - q_i \| + \| q_i q_{k+1} \| \big) \\
&\leq \delta_1 + \cdots + \delta_k + 4k\delta < (k+1)\delta_k 
\end{align*}
to obtain a projection $p_{k+1}$ with $\| p_{k+1} - q_{k+1} \| < \delta_{k+1}$.
We then have $\| p_i - a_i \| < \varepsilon /n \leq\varepsilon$ for all $i=1,\dots ,n$, and 
\[
\bigg\| 1 - \sum_{i=1}^n p_i \bigg\| = \bigg\| \sum_{i=1}^n (a_i - p_i ) \bigg\|
\leq \sum_{i=1}^n \| a_i - p_i \| < n\cdot\frac{\varepsilon}{n} \leq 1 ,
\]
which shows that the projection $1-\sum_{i=1}^n p_i$ must be zero, i.e., $\sum_{i=1}^n p_i = 1$.
\end{proof}

\begin{lemma}\label{L-Cantor mf rf}
Let $r\in \{ 2,3,\dots,\infty \}$. Let $X$ be a zero-dimensional compact metrizable space 
and $F_r \curvearrowright X$ a continuous action.
Suppose that $C(X)\rtimes_\red F_r$ is an MF algebra. Then the action is residually finite.
\end{lemma}

\begin{proof}
Since the MF property passes to $C^*$-subalgebras and residual finiteness is a condition
that is witnessed on finitely many group elements at a time, we may assume that $r$ is finite.
We may also assume that $X$ is infinite, for otherwise the action is automatically residually finite.
Then by metrizability we can take an enumeration $q_1 , q_2 , \dots$ of the projections in $C(X)$.
On $X$ we define the compatible metric 
\[
d(x,y) = \sum_{i=1}^\infty 2^{-j} | q_j (x) - q_j (y) | .
\]

Let $\varepsilon > 0$.
Choose a $n\in\Nb$ such that $2^{-n+1} < \varepsilon$.
Take a partition of unity $\cP = \{ p_1 ,\dots ,p_m \} \subseteq C(X)$ consisting of nonzero projections 
such that (i) each of the projections $q_1 , \dots , q_n$ is a sum of projections in $\cP$, and
(ii) the clopen subsets of $X$ of which the projections in $\cP$ are characteristic
functions all have diameter less than $\varepsilon$.

Write $u_1 , \dots ,u_r$ for the canonical unitaries in the crossed product $C(X)\rtimes_\red F_r$
corresponding to the standard generators $s_1 , \dots , s_r$ of $F_r$. 
Write $\Omega$ for the set of all elements in $C(X)\rtimes_\red F_r$ of the form 
$up$ or $pu^*$ where $p\in\cP\cup \{ 1 \}$ and $u\in \{ 1,u_1 ,\dots ,u_r \}$. Write $A$ for the unital commutative
$C^*$-subalgebra of $C(X)\rtimes_\red F_r$ generated by elements of the form $upu^*$ where
$p\in\cP$ and $u\in \{ 1,u_1 ,\dots ,u_r \}$.
Let $\delta > 0$ be such that $((1+\delta )^2 + (1+\delta ) + 1)\delta < 1/4$.
Let $\delta'$ be a strictly positive number less than $\delta$ to be further specified.
Since $C(X)\rtimes_\red F_r$ is MF, by Proposition~\ref{P-MF} there exist a $d\in\Nb$ and a 
unital $^*$-linear map $\varphi : A \to M_d$ such that 
\begin{enumerate} 
\item $\big| \| \varphi (a) \| - \| a \| \big| < \delta \| a \|$ for all $a\in A$,

\item $\| \varphi (u_k p_i u_k^* u_l p_j u_l^* ) - \varphi (u_k p_i u_k^* )\varphi (u_l p_j u_l^* ) \| < \delta'$ 
for all $i,j=1,\dots ,m$ and $k,l=1,\dots,r$,

\item $\| \varphi (u_k p_i u_k^* ) - \varphi (u_k )\varphi (p_i u_k^* ) \| < \delta$ 
for all $i=1,\dots ,m$ and $k=1,\dots,r$,

\item $\| \varphi (p_i u_k^* ) - \varphi (p_i )\varphi (u_k^* ) \| < \delta$ 
for all $i=1,\dots ,m$ and $k=1,\dots,r$,

\item $\| \varphi (u_k )\varphi (u_k )^* - 1 \| < \delta'$ for all $k=1,\dots,r$.
\end{enumerate}
Since every minimal projection in $A$ has the form $p_{i_0} u_1 p_{i_1} u_1^* \cdots u_r p_{i_r} u_r^*$
for some $i_0 , \dots ,i_r \in \{ 1,\dots ,m \}$, by repeated use of (1) and (2) and the triangle inequality 
we see that by taking $\delta'$ small enough we can arrange for 
$\| \varphi (p)\varphi (p' ) \|$ and $\| \varphi (p)^2 - \varphi (p) \|$ to be as small as we wish
for all distinct minimal projections $p,p' \in A$.
Thus assuming $\delta'$ to be small enough we can construct a unital homomorphism
$\Psi : A\to M_d$ such that $\| \Psi (a) - \varphi (a) \| < \delta \| a \|$ for all $a\in A$ by using Lemma~\ref{L-hom}
to perturb the images of the minimal projections in $A$ under $\varphi$ to pairwise orthogonal projections in $M_d$
and defining $\Psi (p)$ for a minimal projection $p\in A$ to be the perturbation of $\varphi (p)$.
Take a maximal commutative subalgebra $B$ of $M_d$ containing $\Psi (A)$.
By recoordinatizing via an automorphism of $M_d$, we may assume $B$ to be the diagonal subalgebra of $M_d$.

In view of condition (3) above, for each $k=1,\dots ,r$ we can define the unitary
$v_k = \varphi (u_k )/\sqrt{\varphi (u_k )\varphi (u_k )^*}$ in $M_d$ and we may arrange that
$\| \varphi (u_k ) - v_k \| < \delta$ by taking $\delta'$ to be small enough.
For $i=1,\dots ,m$ and $k=1,\dots ,r$ we have
\begin{align*}
\lefteqn{\| \varphi (u_k )\varphi (p_i )\varphi (u_k )^* - v_k \Psi (p_i )v_k^* \|} \hspace*{20mm} \\
\hspace*{20mm} &\leq \| (\varphi (u_k ) - v_k )\varphi (p_i )\varphi (u_k )^* \| 
+ \| v_k (\varphi (p_i ) - \Psi (p_i )) \varphi (u_k )^* \| \\
&\hspace*{30mm} \ + \| v_k \Psi (p_i ) (\varphi (u_k ) - v_k )^* \| \\
&\leq (1+\delta )^2 \| \varphi (u_k ) - v_k \| + (1+\delta )\| \varphi (p_i ) - \Psi (p_i ) \| 
+ \| \varphi (u_k ) - v_k \| \\
&\leq ((1+\delta )^2 + (1+\delta ) + 1)\delta 
< \frac14
\end{align*}
in which case
\begin{align*}
\lefteqn{\| \Psi (u_k p_i u_k^* ) - v_k \Psi (p_i )v_k^* \|}\hspace*{10mm} \\
\hspace*{10mm} &\leq \| \Psi (u_k p_i u_k^* ) - \varphi (u_k p_i u_k^* ) \| 
+ \| \varphi (u_k p_i u_k^* ) - \varphi (u_k )\varphi (p_i u_k^* ) \| \\
&\hspace*{10mm} \ + \| \varphi (u_k ) \| \| \varphi (p_i u_k^* ) - \varphi (p_i )\varphi (u_k^* ) \| 
+ \| \varphi (u_k )\varphi (p_i )\varphi (u_k )^* - v_k \Psi (p_i )v_k^* \| \\
&< \delta + \delta + (1+\delta )\delta + \frac14 
< 1 .
\end{align*}
It follows that $\Psi (u_k p_i u_k^* )$ and $v_k \Psi (p_i )v_k^*$ are homotopic projections,
and thus have the same trace.
But the trace of $v_k \Psi (p_i )v_k^*$ is the same as the trace of $\Psi (p_i )$.
Thus for all $i=1,\dots ,m$ and $k=1,\dots ,r$ the projections $\Psi (u_k p_i u_k^* )$ and $\Psi (p_i )$ have the same trace.
It follows that for each $k=1,\dots ,r$ we can find a permutation matrix
$w_k \in M_k$ such that $w_k \Psi (p_i )w_k^* = \Psi (u_k p_i u_k^* )$ for all $i=1,\dots ,m$.
Note that since $w_k$ is a permutation matrix we have $w_k B w_k^* = B$. 

Write $P(B)$ for the pure state space of $B$, which has $d$ elements.
Let $\omega\in P(B)$. Since $A$ is a unital $C^*$-subalgebra of $C(X)$, by Gelfand theory there exists
a point $x_\omega \in X$ which, when viewed as a pure state $\hat{x}_\omega$ on $C(X)$, restricts to $\omega\circ\Psi$ on $A$.
Set $\zeta (\omega ) = x_\omega$. This defines a map $\zeta : P(B)\to X$. Now for each $i=1,\dots ,m$
the projection $\Phi (p_i )$ is nonzero since $\| \Phi (p_i ) \| \geq \| \varphi (p_i ) \| - \delta \geq \| p_i \| - 2\delta > 0$, 
and so there is an $\omega\in P(B)$ such that $\omega (\Phi (p_i )) = 1$ and hence
$p_i (x_\omega ) = \hat{x}_\omega (p_i) = 1$. It follows that the image of $\zeta$ is $\varepsilon$-dense 
in $X$ since the supports of the projections $p_1 , \dots , p_m$ all have diameter less than $\varepsilon$. 

Let $F_r$ act on $P(B)$ so that for each $k=1,\dots ,r$ the generator $s_k$ acts as the bijection 
$\omega\mapsto\omega\circ\Ad w_k^*$. Then
for $\omega\in P(B)$, $k=1,\dots ,r$, and $i=1,\dots ,m$ we have
\begin{align*}
(s_k \hat{x}_\omega )(p_i )
&= \hat{x}_\omega (\alpha_{s_k}^{-1} (p_i )) 
= \hat{x}_\omega (u_k^* p_i u_k ) \\
&= (\omega\circ\Psi )(u_k^* p_i u_k ) 
= (\omega\circ \Ad w_k^* )(\Psi (p_i )) \\
&= \hat{x}_{s_k \omega} (p_i ) .
\end{align*}
so that $(s_k \hat{x}_\omega )(q_i ) = \hat{x}_{s_k \omega} (q_i )$ for $i=1,\dots ,n$ and hence
\[
d(s_k \zeta (\omega ), \zeta (s_k \omega )) = d(s_k \hat{x}_\omega , \hat{x}_{s_k \omega} ) 
= \sum_{i=1}^\infty 2^{-i} |(s_k \hat{x}_\omega )(q_i ) - \hat{x}_{s_k \omega} (q_i ) | 
\leq \frac{1}{2^{n-1}} < \varepsilon .
\]
Since $\varepsilon$ was an arbitrary positive number we conclude that the action is residually finite.
\end{proof}

\begin{theorem}\label{T-mf}
Let $r\in \{ 2,3,\dots,\infty \}$. Let $X$ be a zero-dimensional compact metrizable space and $F_r \curvearrowright X$ a continuous action.
Then the action is residually finite if and only if $C(X)\rtimes_\red F_r$ is an MF algebra.
\end{theorem}

\begin{proof}
Suppose that the action is residually finite. Then by Proposition~\ref{P-rf qd} the action is quasidiagonal,
and since $C^*_\red (F_r )$ is an MF algebra \cite{HaaTho05} we infer by Theorem~\ref{T-qd mf}
that $C(X)\rtimes_\red F_r$ is an MF algebra.
The other direction is Lemma~\ref{L-Cantor mf rf}.
\end{proof}

Note that for $r=1$ the conclusion of Theorem~\ref{T-mf} is valid without the zero-dimensionality hypothesis
by Pimsner's result from \cite{Pim83} (see Section~\ref{S-integers}).

\section{Paradoxical decompositions}\label{S-pd}

For a compact Hausdorff space $X$ we write $\sC_X$ for the collection of clopen subsets of $X$
and $\sB_X$ for the collection of Borel subsets of $X$.

\begin{definition}
Suppose that $G$ acts on a set $X$. Let $\sS$ be a collection of subsets of $X$. Let $k$ and $l$ be 
integers with $k>l>0$. We say that a set $A\subseteq X$ is {\it $(G,\sS ,k,l)$-paradoxical} (or simply
{\it $(G,\sS )$-paradoxical} when $k=2$ and $l=1$) if there exist 
$A_1 , \dots , A_n \in\sS$ and $s_1 ,\dots , s_n \in G$ such that
$\sum_{i=1}^n 1_{A_i} \geq k\cdot 1_A$ and $\sum_{i=1}^n 1_{s_i A_i} \leq l\cdot 1_A$.
The set $A$ is said to be {\it completely $(G,\sS )$-nonparadoxical} if it fails to be
$(G,\sS ,k,l)$-paradoxical for all integers $k>l>0$.
\end{definition}

\begin{remark}
Suppose that $\sS$ is actually a subalgebra of the power set $\sP_X$, which will always be the case in our applications. 
Then we can express the $(G,\sS ,k,l)$-paradoxicality of a set $A$ in $\sS$ by partitioning copies of $A$ instead of
merely counting multiplicities. More precisely, $A$ is $(G,\sS ,k,l)$-paradoxical if and only if
for each $i=1,\dots ,k$ there exist an $n_i \in\Nb$ and 
$A_{i,1} , \dots , A_{i,n_i} \in\sS$, $s_{i,1} ,\dots , s_{i,n_i} \in G$, and 
$m_{i,1} , \dots , m_{i,n_i} \in \{ 1,\dots ,l \}$
so that $\bigcup_{j=1}^{n_i} A_{i,j} = A$ for each $i=1,\dots ,k$ and
the sets $s_{i,j} A_{i,j} \times \{ m_{i,j} \} \subseteq A\times \{ 1,\dots , l \}$ for $j=1,\dots ,n_i$ and $i=1,\dots ,k$ 
are pairwise disjoint. For the nontrivial direction, observe that if $A_1 , \dots , A_n$ and $s_1 ,\dots , s_n$ 
are as in the definition of $(G,\sS ,k,l)$-paradoxicality then the sets of the form
\[
A \cap \bigg(\bigg( \bigcap_{i\in P} A_i \bigg) \setminus \bigcup_{i\in \{ 1,\dots ,n\} \setminus P} A_i \bigg)  
\cap s_j^{-1} \bigg(\bigg( \bigcap_{i\in Q} s_i A_i \bigg) \setminus \bigcup_{i\in \{ 1,\dots ,n\} \setminus Q} s_i A_i \bigg) ,
\]
where $P$ and $Q$ are nonempty subsets of $\{ 1,\dots ,n\}$ with $|P|\leq k$ and $|Q|\leq l$ and 
$j\in P$, can be relabeled so as to produce the desired $A_{i,j}$.
\end{remark}

Suppose that $G$ acts on a set $X$. Let $\sS$ be a $G$-invariant subalgebra of the power set $\sP_X$ of $X$.
The {\it type semigroup $S(X,G,\sS )$} of the action with respect to $\sS$ is the preordered semigroup  
\[
\bigg\{ \bigcup_{i\in I} A_i \times \{ i \} : I\text { is a finite subset of } \Nb \text{ and } A_i \in\sS 
\text{ for each } i\in I \bigg\} \Big/ \sim
\]
where $\sim$ is the equivalence relation under which $P = \bigcup_{i\in I} A_i \times \{ i \}$ is equivalent to
$Q = \bigcup_{i\in J} B_i \times \{ i \}$ if there exist $n_i , m_i \in\Nb$,
$C_i \in\sS$, and $s_i \in G$ for $i=1,\dots ,k$ such that
$P = \bigsqcup_{i=1}^k C_i \times \{ n_i \}$ and $Q = \bigsqcup_{i=1}^k s_i C_i \times \{ m_i \}$.
Addition is defined by
\[
\bigg[ \bigcup_{i\in I} A_i \times \{ i \} \bigg] + \bigg[ \bigcup_{i\in J} B_i \times \{ i \} \bigg]
= \bigg[ \bigg( \bigcup_{i\in I} A_i \times \{ i \} \bigg) \cup \bigg( \bigcup_{i\in J + \max I} B_i \times \{ i \} \bigg) \bigg] ,
\]
and for the preorder we declare that $a\leq b$ if $b=a+c$ for some $c$.

The following is a standard observation. 

\begin{lemma}\label{L-measure comp nonp}
Let $X$ be a compact metrizable space and $G\curvearrowright X$ a continuous action.
Let $B$ be a nonempty Borel subset of $X$. Suppose that there is a $G$-invariant Borel probability measure $\mu$ on 
$X$ with $\mu (B) > 0$. Then $B$ is completely $(G,\sB_X )$-nonparadoxical.
\end{lemma}

\begin{proof}
Let $\mu$ be a $G$-invariant Borel probability measure on $X$ with $\mu (B) > 0$.
Suppose that $B$ fails to be completely $(G,\sB_X )$-nonparadoxical. Then there are $k,l\in\Nb$ with $k > l$
for each $i=1,\dots ,k$ there are an $n_i \in\Nb$ and 
$B_{i,1} , \dots , B_{i,n_i} \in\sB_X$, $s_{i,1} ,\dots , s_{i,n_i} \in G$, and 
$m_{i,1} , \dots , m_{i,n_i} \in \{ 1,\dots ,l \}$
so that $\bigcup_{j=}^{n_i} B_{i,j} = B$ for each $i=1,\dots ,k$ and the sets 
$s_{i,j} B_{i,j} \times \{ m_{i,j} \} \subseteq B\times \{ 1,\dots , l \}$ for $j=1,\dots ,n_i$ and $i=1,\dots ,k$ 
are pairwise disjoint. Since $\mu$ is $G$-invariant we have
\begin{align*}
k\mu (B) \leq \sum_{j=1}^k \sum_{i=1}^n \mu (B_i ) = \sum_{j=1}^k \sum_{i=1}^n \mu (s_{i,j} B_i )
= \mu \bigg( \bigcup_{j=1}^k \bigcup_{i=1}^n \mu (s_{i,j} B_i )\bigg) \leq l\mu (B) ,
\end{align*}
and dividing by $\mu (B)$ yields $k\leq l$, a contradiction.
We conclude that $B$ is completely $(G,\sB_X )$-nonparadoxical.
\end{proof}

\begin{lemma}\label{L-comp nonp measure}
Let $G\curvearrowright X$ be a continuous action on a zero-dimensional compact metrizable space.
Let $A$ be a completely $(G,\sC_X )$-nonparadoxical clopen subset of $X$ such that $G\cdot A = X$. 
Then there exists a $G$-invariant Borel probability measure $\mu$ on $X$ such that $\mu (A) > 0$.
\end{lemma}

\begin{proof}
We claim that there is a state $\sigma$ on $S(X,G,\sC_X )$ such that $\sigma ([A])>0$. Indeed suppose that this is not
the case. Since $G\cdot A = X$, by compactness there is a finite set $F\subset G$ 
such that $F\cdot A = X$ and hence $[X] \leq |F|[A]$. 
By the Goodearl-Handelman theorem \cite[Lemma 4.1]{GooHan76}
we can find $n,m\in\Nb$ such that $m > n|F|$ and $m[A]\leq n[X]$, in which case $m[A] \leq n|F|[A]$, 
contradicting the complete $(G,\sC_X )$-nonparadoxicality of $A$. 
Thus the desired $\sigma$ exists.

By Lemma~5.1 of \cite{RorSie10} there is a $G$-invariant Borel measure 
$\nu$ on $X$ such that $\nu (B) = \sigma ([B])$ for all clopen sets $B\subseteq X$. 
Since $\nu (X) \geq \nu (V) > 0$ and
\begin{align*}
\nu (X) = \nu (F\cdot A) \leq \bigcup_{s\in F} \nu (sA) = |F|\nu (A) < \infty ,
\end{align*}
we can set $\mu (\cdot ) = \nu (\cdot ) / \nu (X)$ to obtain a $G$-invariant Borel probability measure on $X$.
\end{proof}

\begin{proposition}\label{P-measure minimal Cantor}
Let $G\curvearrowright X$ be a minimal continuous action on a zero-dimensional compact metrizable space. 
Then the following are equivalent:
\begin{enumerate}
\item there is a $G$-invariant Borel probability measure on $X$,

\item $X$ is completely $(G,\sC_X )$-nonparadoxical,

\item there is a nonempty clopen subset of $X$ which is completely $(G,\sC_X )$-nonparadoxical.
\end{enumerate}
\end{proposition}

\begin{proof}
Lemma~\ref{L-measure comp nonp} yields (1)$\Rightarrow$(2), while (2)$\Rightarrow$(3) is trivial.
Since for any clopen set $A\subseteq X$ we have $G\cdot A = X$ by minimality, 
we obtain (3)$\Rightarrow$(1) from Lemma~\ref{L-comp nonp measure}.
\end{proof}

\begin{proposition}\label{P-compress}
Let $\alpha$ be a continuous action of $G$ on a zero-dimensional compact metrizable space $X$
such that $C(X)\rtimes_\red G$ is stably finite.
Then every nonempty clopen subset of $X$ is completely $(\alpha ,\sC_X )$-nonparadoxical.
\end{proposition}

\begin{proof}
Write $A=C(X)\rtimes_\red G$ for economy.
Suppose that there is a nonempty clopen subset $V$ of $X$ which fails to
be completely $(\alpha ,\sC_X )$-nonparadoxical. 
Then there are $k,l\in\Nb$ with $k > l$, a clopen partition $\{ V_1 , \dots , V_n \}$ of $V$,
and $s_{i,j} \in G$ and $m_{i,j} \in \{ 1,\dots ,l \}$ for $i=1,\dots ,n$ and $j=1,\dots ,k$ such that  
the sets $s_{i,j} V_i \times \{ m_{i,j} \} \subseteq V\times \{ 1,\dots , l \}$ for $i=1,\dots ,n$ and $j=1,\dots ,k$
are pairwise disjoint. For $i=1,\dots ,n$ and $j=1,\dots ,k$ set
\[
a_{i,j} = e_{m_{i,j} ,j} \otimes u_{s_{i,j}} 1_{V_i} \in M_k \otimes A 
\]
where the first component in the elementary tensor is a standard matrix unit in $M_k$. 
Then $a_{ij}$ is a partial isometry with
\begin{align*}
(a_{i,j} )^* a_{i,j} &= e_{jj} \otimes 1_{V_i} , \\
a_{i,j} (a_{i,j} )^* &= e_{m_{i,j} m_{i,j}} \otimes 1_{s_{i,j} V_i} .
\end{align*}
Now since
\[
\sum_{j=1}^k \sum_{i=1}^n (a_{i,j} )^* a_{i,j} = \sum_{j=1}^k e_{jj} \otimes 1_V = 1\otimes 1_V ,
\]
in $K_0 (A)$ we have
$\sum_{j=1}^k \sum_{i=1}^n [(a_{i,j} )^* a_{i,j} ] = k[1_V ]$.
On the other hand, 
\[
\sum_{j=1}^k \sum_{i=1}^n a_{i,j} (a_{i,j} )^* \leq \bigg( \sum_{j=1}^l e_{jj} \bigg) \otimes 1_V ,
\]
so that in $K_0 (A)$ we have $k[1_V ] \leq l[1_V ]$ and hence $-[1_V] \geq (l-k-1) [1_V] \geq 0$.
Thus there exists a projection $p$ in some matrix algebra over $A$ such that $[p] + [1_V ] = 0$.
We can thus find a $d\in\Nb$ and pairwise orthogonal projections 
$p',q,r\in M_d \otimes A$ such that 
\begin{enumerate}
\item $p'$ and $q$ are Murray-von Neumann equivalent to $p$ and $1_V$, respectively, 
in some matrix algebra over $A$ (viewing $1_V$ and $p$ as sitting in arbitrarily large
matrix algebras over $A$ as consistent with the definition of $K_0$), and

\item $p'+q+r$ is Murray-von Neumann equivalent to $r$ in $M_d \otimes A$. 
\end{enumerate}
Since $1_V \neq 0$ the projection
$q$ is not equal to zero, and so $p'+q+r$ is an infinite projection.
This means that $C(X)\rtimes_\red G$ fails to be stably finite, contradicting (1).
\end{proof}

Finally we discuss the relation between residual finiteness and complete nonparadoxicality
in the zero-dimensional setting.

\begin{proposition}\label{P-rf nonpd}
Let $G\curvearrowright X$ be a residually finite continuous action on a zero-dimensional 
compact metrizable space. Then every nonempty clopen subset of $X$ is completely $(G,\sC_X )$-nonparadoxical.
\end{proposition}

\begin{proof}
Suppose to the contrary that there is a nonempty clopen set $A\subseteq X$ which is $(G,\sC_X ,k,l)$-nonparadoxical
for some integers $k>l>0$. Then there exist nonempty clopen sets 
$A_1 , \dots , A_n \subseteq A$ and $s_1 ,\dots , s_n \in G$ such that
$\sum_{i=1}^n 1_{A_i} \geq k\cdot 1_A$ and $\sum_{i=1}^n 1_{s_i A_i} \leq l\cdot 1_A$.
Let $d$ be a compatible metric on $X$. Take an $\varepsilon > 0$ which is smaller than the Hausdorff
distance between $A_i$ and $X\setminus A_i$ for each $i=1,\dots ,n$ and smaller that the Hausdorff distance
between $s_i A_i$ and $X\setminus s_i A_i$ for each $i=1,\dots ,n$.
By residual finiteness there is a finite set $E$, an action of $G$ on $E$, and a map $\zeta : E\to X$ such that
$\zeta (E)$ is $\varepsilon$-dense in $X$ and $d(\zeta (s_i z), s_i \zeta (z)) < \varepsilon$ 
for all $z\in E$ and $i=1,\dots ,n$. Since
$\zeta (E)$ is $\varepsilon$-dense in $X$, $\zeta^{-1} (A)$ is nonempty. 
For each $i=1,\dots ,n$ set $E_i = \zeta^{-1} (A_i )$. 

Each element of $\zeta^{-1} (A)$ is contained in $\zeta^{-1} (s_i A_i )$ for at most
$l$ values of $i$, and thus, since $\zeta (s_i E_i )\subseteq s_i A_i$ for $i=1,\dots ,n$ by our choice of $\varepsilon$, 
\[
l|\zeta^{-1} (A)| \geq \sum_{i=1}^n |\zeta^{-1} (s_i A_i)| \geq \sum_{i=1}^n |s_i E_i | = \sum_{i=1}^n |E_i | .
\]
On the other hand, each element of $\zeta^{-1} (A)$ is contained in $\zeta^{-1} (A_i )$ for at least
$k$ values of $i$, so that $\sum_{i=1}^n |E_i | \geq k|\zeta^{-1} (A)|$.
Therefore $l\geq k$, a contradiction.
\end{proof}

Residual finiteness is not a necessary condition for the conclusion in Proposition~\ref{P-rf nonpd}, 
as can be seen for $G=\Zb$ as follows.
Recall that the group $G$ is said to be {\it supramenable} if for every nonempty set $A\subseteq G$
there is a finitely additive left-invariant measure $\mu : \sP_X \to [0,+\infty]$ with $\mu (A) = 1$ \cite{Ros74}
(see Chapter~10 of \cite{Wag95}).
Every group of polynomial growth, in particular $\Zb$, is supramenable \cite{Ros74}. 
Supramenability is easily seen to be equivalent to the property that for every action of $G$ on a set $X$
and every nonempty set $A\subseteq X$ there is a finitely additive $G$-invariant measure $\mu : \sP_X \to [0,+\infty]$ 
with $\mu (A) = 1$. It follows by a standard observation as in the proof of Lemma~\ref{L-measure comp nonp}
that, for every action of a supramenable $G$ on a set $X$, every subset of $X$ is completely
$(G,\sP_X )$-nonparadoxical, where $\sP_X$ is the power set of $X$. 
On the other hand, for $\Zb$-actions residual finiteness is the same as chain recurrence 
(see the beginning of Section~\ref{S-integers}), and one can easily construct a
continuous $\Zb$-action on the Cantor set which fails to be chain recurrent using Lemma~2 of \cite{Pim83},
which shows that if $X$ a compact Haudorff space and $T:X\to X$ is a homeomorphism then $T$ is chain recurrent 
if and only if there does not exist an open set $U\subseteq X$ such
that $U\setminus T(\overline{U} )\neq \emptyset$ and $T(\overline U) \subseteq U$.

\section{Minimal actions of $F_r$}\label{S-minimal}

Here we give several characterizations of residual finiteness for minimal actions of the free group $F_r$ 
where $r\in\Nb\cup \{ \infty \}$. This is done by stringing together a variety of known results with a few of the
observations from the previous sections.

The following lemma is essentially extracted from the proof of Proposition~2.3 in \cite{LubSha04}. 
Note that in that proof it is claimed that, given a continuous action of $F_r$ on compact metric space $X$ 
preserving a Borel probability measure, for every $\varepsilon$ one can partition $X$ into
finitely many measurable subsets of equal (and hence rational) measure and diameter less than $\varepsilon$. 
This is false in general, as $X$ could contain a clopen subset $A$ of irrational measure and such an
$A$ would be equal to a union of sets in the partition whenever $\varepsilon$ is smaller than the distance
between $x$ and $y$ for every $x\in A$ and $y\in X\setminus A$. The argument can be modified to take
of this problem, but we will instead give a simpler proof that was supplied to us by Hanfeng Li. 

\begin{lemma}\label{L-invt measure res fin}
Let $X$ be a compact metrizable space and $F_r \curvearrowright X$ a continuous action.
Suppose there exists an $F_r$-invariant Borel probability measure $\mu$ on $X$ with full support. 
Then the action is residually finite.
\end{lemma}

\begin{proof}
We may assume that $r$ is finite in view of the definition of residual finiteness. Write $S$ for the
standard generating set for $F_r$.
Let $\varepsilon > 0$. 
Take a finite measurable partition $\cP$ of $X$ whose elements have nonzero measure and diameter
less than $\varepsilon$. 
Write $\cQ$ for the collection of sets in the join $\bigvee_{s\in S} s\cP$ which have nonzero measure.
Consider, for each $P\in\cP$ and $s\in S$, the homogeneous linear equation 
$\sum_{Q\in\cQ , Q\subseteq P} x_Q = \sum_{Q\in\cQ , Q\subseteq sP} x_Q$ in the variables $x_Q$ for $Q\in \cQ$.
This system of equations has the solution $x_Q = \mu (Q)$ for $Q\in\cQ$, and, since the
rational solutions are dense in the set of real solutions by virtue of the rationality of the coefficients,
we can find a solution consisting of rational $x_Q$ which are close enough to the corresponding quantities
$\mu (Q)$ to be all nonzero. Pick a positive integer $M$ such that $Mx_Q$ is an integer for every $Q\in\cQ$.
For each $Q\in \cQ$ take a set $E_Q$ of cardinality $Mx_Q$ and define $E$ to be the disjoint union of these sets.
Take a map $\zeta : E\to X$ which sends $E_Q$ into $Q$ for each $Q\in\cQ$. Now for every $P\in\cP$
and $s\in S$ the sets $\bigcup_{Q\in\cQ , Q\subseteq P} E_Q$ and $\bigcup_{Q\in\cQ , Q\subseteq sP} E_Q$
have the same cardinality and so we can define an action of $F_r$ on $E$ by having a generator
$s$ send $\bigcup_{Q\in\cQ , Q\subseteq P} E_Q$ to $\bigcup_{Q\in\cQ , Q\subseteq sP} E_Q$ in some 
arbitrarily chosen way for each $P\in\cP$. Then $\zeta$ and this action on $E$ witness the definition of 
residual finiteness with respect to $\varepsilon$ and the generating set $S$.
\end{proof}

In \cite{CunPed79} Cuntz and Pedersen introduced, for an action $\alpha$ of $G$ on a $C^*$-algebra $A$, a notion of {\it G-finiteness},
which means that there do not exist distinct elements $a,b\in A$ such that (i) $a\leq b$ and (ii) $a$ and $b$ are $G$-equivalent 
in the sense that there exist a collection of elements $u_{i,s_i} \in A$, where $i$ ranges in an arbitrary index set $I$
and each $s_i$ is an element of $G$, such that 
\[
a = \sum_{i\in I} u_{i,s_i}^* u_{i,s_i} \hspace*{5mm} \text{and} \hspace*{5mm} b = \sum_{i\in I} \alpha_{s_i} (u_{i,s_i} u_{i,s_i}^* )
\]
where the sums are norm convergent. Theorem~8.1 of \cite{CunPed79} asserts that $A$ is $G$-finite if and only if 
it has a separating family of $G$-invariant tracial states, which is equivalent to the existence of a faithful
$G$-invariant tracial state if $A$ is separable.
We will say that a continuous action $G\curvearrowright$ on a compact Hausdorff space $X$
is {\it $G$-finite (in the Cuntz-Pedersen sense)} if the induced action on $C(X)$ is $G$-finite.

\begin{theorem}\label{T-minimal}
Let $X$ be a compact metrizable space and $F_r \curvearrowright X$ a minimal continuous action. 
Then the following are equivalent:
\begin{enumerate}
\item the action is residually finite,

\item there is an $F_r$-invariant Borel probability measure on $X$,

\item $C(X)\rtimes_\red F_r$ is an MF algebra,

\item $C(X)\rtimes_\red F_r$ is stably finite,

\item the action is $F_r$-finite in the Cuntz-Pedersen sense,
\end{enumerate}
If moreover $X$ is zero-dimensional then we can add the following conditions to the list:
\begin{enumerate}
\item[(6)] every nonempty clopen subset of $X$ is completely $(F_r ,\sC_X )$-nonparadoxical. 

\item[(7)] there exists a nonempty clopen subset of $X$ which is completely $(F_r ,\sC_X )$-nonparadoxical. 
\end{enumerate}
\end{theorem}

\begin{proof}
(1)$\Rightarrow$(2). By Proposition~\ref{P-res fin measure}.
 
(2)$\Rightarrow$(1). By minimality every $G$-invariant Borel probability measure on $X$ has full
support, and so Lemma~\ref{L-invt measure res fin} applies.

(1)$\Rightarrow$(3). By Theorem~\ref{T-qd mf}.

(3)$\Rightarrow$(4). By Proposition~3.3.8 of \cite{BlaKir97}.

(4)$\Rightarrow$(2). Stable finiteness implies the existence of a quasitrace (see Section~V.2 of \cite{Bla06}) and restricting 
a quasitrace on $C(X)\rtimes_\red F_r$ to $C(X)$ yields a $G$-invariant Borel probability measure on $X$.

(2)$\Leftrightarrow$(5). Apply Theorem~8.1 of \cite{CunPed79} as quoted above, using the fact that
every Borel probability measure on $X$ has full support by the minimality of the action.

Finally, in the case that $X$ is zero-dimensional 
(2)$\Leftrightarrow$(6)$\Leftrightarrow$(7) is a special case of Proposition~\ref{P-measure minimal Cantor}. 
\end{proof}

The hypotheses in the above theorem actually imply conditions (1) to (5) in the case of $\Zb$, i.e., when $r=1$. 
When $r\geq 2$ these conditions may fail, as the action on the Gromov boundary (or any amenable minimal action) shows.

\begin{problem}
Suppose in the above theorem that the action is topologically free and $X$ is Cantor set. It then follows by
Theorem~5.4 of \cite{RorSie10} that if every nonempty clopen subset of $X$ is $(F_r ,\sC_X )$-paradoxical
then $C(X)\rtimes_\red F_r$ is purely infinite. We therefore ask whether there is a dichotomy between the MF property
and pure infiniteness within this class of actions. This would be the case if for any such action 
$(F_r ,\sC_X )$-nonparadoxicality implies complete $(F_r ,\sC_X )$-nonparadoxicality for every nonempty clopen subset of $X$,
i.e., if the type semigroup $S(X,G,\sC_X )$ is almost unperforated in the sense of Section~5 of \cite{RorSie10}.
\end{problem}

\begin{example}
Let $\alpha$ be a minimal continuous action of $F_r$ on $\Tb$. 
Then the following are equivalent.
\begin{enumerate}
\item $\alpha$ is residually finite,

\item there is a $G$-invariant Borel probability measure on $\Tb$,

\item $C(\Tb )\rtimes_\red F_n$ is an MF algebra,

\item $C(\Tb )\rtimes_\red F_n$ is stably finite,

\item $\alpha_s$ is residually finite for every $s\in G$.
\end{enumerate}
The equivalences (1)$\Leftrightarrow$(2)$\Leftrightarrow$(3)$\Leftrightarrow$(4) are by Theorem~\ref{T-minimal},
(1)$\Rightarrow$(5) is trivial, and (5)$\Rightarrow$(2) is a consequence of a result of Margulis \cite{Mar00}.  
Note that, by Theorem~2 of \cite{Mar00}, if a minimal action of a countable group on the circle has
an invariant Borel probability measure then it is conjugate to an isometric action with respect to the standard metric
and thus factors through a group containing a commutative subgroup of index at most two.
Therefore there are no faithful residually finite continuous actions of $F_2$ on the circle.
\end{example}

\section{Actions on spaces of probability measures}\label{S-affine}

In this section we will use spaces of probability measures
in order to construct an example, for every nonamenable $G$, of a continuous action of $G$ on a
compact metrizable space such that $\alpha$ is not residually finite but its restriction to every
cyclic subgroup of $G$ is residually finite. The crossed product of such an action fails to be stably finite,
but, unlike in the case of integer actions, one cannot witness this failure 
by using the compression of an open set by a single group element to construct a nonunitary isometry.
See the discussion after Proposition~\ref{P-nonamenable}.

For a compact Hausdorff space $X$ we will view the space $M_X$ of Borel probability measures
on $X$ with its weak$^*$ topology as a topological subspace of $M_X$ via the point mass identification.

\begin{proposition}\label{P-affine}
Let $X$ be a compact metrizable space and $G\curvearrowright X$ a residually finite continuous action.
Then the induced action $G\curvearrowright M_X$ is residually finite.
\end{proposition}

\begin{proof}
Let $\Omega$ be a finite subset of the unit ball of $C(X)$,
and equip $M_X$ with the continuous pseudometric 
$d_\Omega (\sigma , \omega ) = \max_{f\in\Omega} | \sigma (f) - \omega (f) |$. 
Let $\varepsilon > 0$.
By residual finiteness there are a finite set $Y$, an action of $G$ on $Y$, and a map $\zeta : Y\to X$
such that (i) $d_\Omega (s\zeta (y),\zeta (sy)) < \varepsilon$ for all $y\in Y$ and $s\in F$ and (ii)
$\zeta (Y)$ is $\varepsilon$-dense in $X$ with respect to $d_\Omega$.
Let $m\in\Nb$.
Let $E$ be the finite subset of $M_Y$ consisting of all convex combinations of the form 
$\sum_{y\in Y} c_y \delta_y$ where $c_y \in \{ 0,1/m,2/m,\dots ,(m-1)/m,1\}$
for each $y\in Y$. The the action of $G$ on $Y$ extends to an action of $G$ on $E$ by setting
\[
s \bigg(\sum_{y\in Y} c_y \delta_y \bigg) = \sum_{y\in Y} c_y \delta_{sy} 
\]
for $s\in G$. Define $\tilde{\zeta} : E\to M_X$ as the restriction of the push-forward map $M_Y \to M_X$ associated
to $\zeta$. Then for $z = \sum_{y\in Y} c_y \delta_y \in E$ and $s\in F$ we have
\begin{align*}
d_\Omega (s\tilde{\zeta} (z),\tilde{\zeta} (sz))
&= \sup_{f\in\Omega} \bigg| f\bigg( \sum_{y\in Y} c_y s\zeta (y)\bigg) - f\bigg( \sum_{y\in Y} c_y \zeta (sy)\bigg) \bigg| \\
&\leq \sup_{f\in\Omega} \sum_{y\in Y} c_y | f(s\zeta (y)) - f(\zeta (sy)) | < \varepsilon .
\end{align*}
Since the set of convex combinations of point masses is weak$^*$ dense in $M_X$, given any $\eta > 0$
we can take $\varepsilon$ sufficiently small and $m$ sufficiently large to guarantee that the image of 
$\tilde{\zeta}$ is $\eta$-dense in $M_X$ with respect to $d_\Omega$.
Since the pseudometrics of the 
form $d_\Omega$ generate the uniformity on $M_X$, we conclude that the action $G \curvearrowright M_X$ 
is residually finite.
\end{proof}

\begin{proposition}\label{P-single fixed}
Let $K$ be a compact convex subset of a locally
convex topological vector space and let $T:K\to K$ be a homeomorphism. Then the action of $\Zb$ generated by $T$
is residually finite. 
\end{proposition}

\begin{proof}
Let $\varepsilon > 0$.
Let $\Omega$ be a finite set of continuous affine real-valued functions on $K$,
and equip $K$ with the continuous pseudometric $d_\Omega (x,y) = \max_{f\in\Omega} | f(x) - f(y) |$. 
Let $\varepsilon > 0$. Take a finite set $V\subseteq K$ which is $\varepsilon$-dense for $d_\Omega$. 
Choose an $m\in\Nb$ such that $2m^{-1} \max_{f\in\Omega} \| f \| < \varepsilon$. 
Set $E = \{ (v,k) : v\in V \text{ and } k=-m,\dots ,m\}$ and define the bijection $S : E\to E$ by setting
$S(v,k) = (v,k+1)$ for $v\in V$ and $k=-m,\dots ,m-1$ and $S(v,m) = -m$ for $v\in V$.
 
Since $\Zb$ is amenable there is a $T$-fixed point $w\in K$. Define a map $\zeta : E\to K$
by setting, for each $v\in V$ and $k=-m,-m+1,\dots ,m$,
\[ 
\zeta (v,k) = \bigg( 1 - \frac{|k|}{m} \bigg) T^k v + \frac{|k|}{m} w . 
\]
Then for $v\in V$ and $k=-m,\dots ,m-1$ we have
\begin{align*}
d_\Omega (\zeta (S(v,k)),T\zeta (v,k) ) &= d_\Omega (\zeta (v,k+1) , T\zeta (v,k)) \\
&= \max_{f\in\Omega} \frac1m \big| f(T^{k+1} v) - f(w) \big| \\
&\leq \frac2m \max_{f\in\Omega} \| f \| < \varepsilon , 
\end{align*}
while $\zeta (S(v,m)) = w = T\zeta (v,m)$ so that $d_\Omega (\zeta (S(v,m)),T\zeta (v,m)) = 0$.
Observe that by the uniform continuity of $T$ we can enlarge $m$ to obtain the same kind of estimates for finitely many powers of $T$
at a time. Since the continuous pseudometrics of the form $d_\Omega$ generate the topology on $K$,
we conclude that $T$ is residually finite.
\end{proof}

\begin{proposition}\label{P-measure}
Let $G \curvearrowright X$ be a continuous action on a compact metrizable space. 
Suppose that the induced action $G \curvearrowright M_X$ is residually finite.
Then there exists a $G$-invariant Borel probability measure on $X$.
\end{proposition}

\begin{proof}
By Proposition~\ref{P-res fin measure} there is a $G$-invariant Borel probability measure $\mu$ on $M_X$.
Consider the barycentre $b_\mu$ of $\mu$, i.e., the unique element of $K$ satisfying 
$b_\mu (f) = \mu (f)$ for all $f\in C_\Rb (M_X )$, with $f$ viewed 
on the right side of the equality as an affine function on $M_K$ under 
the canonical identification. Then $b_\mu$ is a fixed point for the action $G \curvearrowright M_X$,
that is, $b_\mu$ is a $G$-invariant Borel probability measure on $X$.
\end{proof}

\begin{proposition}\label{P-nonamenable}
Suppose that $G$ is nonamenable. Then there exists a continuous action $\alpha$ of $G$ on a
compact metrizable space such that $\alpha$ is not residually finite but its restriction to every
cyclic subgroup of $G$ is residually finite.
\end{proposition}

\begin{proof}
Since $G$ is nonamenable there is a continuous action $\alpha$ of $G$ on a compact metrizable space $X$
which does not admit an invariant Borel probability measure. By Proposition~\ref{P-measure}, the induced
action $\tilde{\alpha}$ of $G$ on $M_X$ is not residually finite. However, for every cyclic subgroup $H$
of $G$ the restriction of $\tilde{\alpha}$ to $H$ is residually finite, which is obvious if $|H| < \infty$
and follows from Proposition~\ref{P-single fixed} otherwise.
\end{proof}

Note that if $G$ is nonamenable and $G\curvearrowright X$ 
is an action as in the statement of Proposition~\ref{P-nonamenable} then $C(X)\rtimes_\red G$ fails to be stably finite,
for otherwise it would admit a quasitrace (see Section~V.2 of \cite{Bla06}) and restricting a quasitrace to $C(X)$ would produce a 
$G$-invariant Borel probability measure on $X$. Therefore $M_k (C(X)\rtimes_\red G)$ contains a nonunitary isometry for some
$k\in\Nb$. However one cannot verify the existence of such an isometry by using the compression of an open
set by a single group element as in the case of crossed products by $\Zb$ \cite{Pim83} (see the discussion preceding 
Theorem~\ref{T-Pimsner} in the next section).

\begin{problem}
For an action of the kind just described, exhibit a nonunitary isometry in some matrix algebra over $C(X)\rtimes_\red G$.
\end{problem}

\begin{problem}
In the statement of Proposition~\ref{P-nonamenable}, can the action be taken to be minimal and/or the space 
taken to be the Cantor set?
\end{problem}

\section{$\Zb$-systems}\label{S-integers}

Here we specialize our discussion to $\Zb$-actions, for which the 
$C^*$-algebraic notions of quasidiagonality, strong quasidiagonality, and 
strong NFness all admit direct dynamical interpretations in terms of 
residual finiteness. Tacit use will be made of the fact that, 
due to the amenability of $\Zb$, the reduced crossed product associated to 
a continuous action of $\Zb$ on a compact Hausdorff space is nuclear and 
coincides canonically with the full crossed product.

For terminological economy we will conceive of a $\Zb$-action
as a pair $(X,T)$ where $X$ is a compact Hausdorff space and $T:X\to X$ is a homeomorphism associated 
to the generator $1$, and we will call such a pair a {\it $\Zb$-system}. 
We will thus speak of residually finite $\Zb$-systems.
A $\Zb$-system $(X,T)$ is said to be {\it metrizable} if $X$ is metrizable.
We write $\orb (x)$ for the orbit of the point $x$ under $T$, i.e., $\orb (x) = \{ T^n x : n\in\Zb \}$.

For $\Zb$-systems, residual finiteness coincides with Conley's concept
of chain recurrence \cite{Con78}, defined as follows.
Let $(X,T)$ be a $\Zb$-system. Let $x,y\in X$ and let $\varepsilon$ be a neighbourhood
of the diagonal in $X\times X$.
An {\it $\varepsilon$-chain} from $x$ to $y$ is a finite sequence
$\{ x_1 = x, x_2 , \dots , x_n = y \}$ in $X$ such that $n>1$ and $(Tx_i , x_{i+1} )
\in\varepsilon$ for every $i=1, \dots , n-1$. The point $x$ is said to be
{\it chain recurrent} if for every neighbourhood of the diagonal in $X\times X$
there is an $\varepsilon$-chain from $x$ to itself (if $X$ is a metric space then
we can quantify instead over all $\varepsilon > 0$, with 
$\varepsilon$-chain taking its meaning in relation to the set 
$\{ (x,y)\in X\times X : d(x,y) < \varepsilon \}$). This is 
equivalent to $x$ being pseudo-nonwandering in the sense of \cite{Pim83}. We
remark that the set of chain recurrent points is a closed 
$T$-invariant subset of $X$. Finally, the system 
$(X,T)$ is said to be {\it chain recurrent} if every point in $X$ is chain 
recurrent. Note that a $\Zb$-system for which there is a dense set of recurrent points is chain recurrent.

\begin{proposition}\label{P-rf cr}
A $\Zb$-system $(X,T)$ is residually finite if and only if it is chain recurrent.
\end{proposition}

\begin{proof}
Suppose first that $(X,T)$ is chain recurrent. Let $F$ be a finite subset of $\Zb$. 
Let $\varepsilon$ and $\varepsilon'$ be neighbourhoods of the diagonal
in $X\times X$. Take a finite $\varepsilon$-dense subset $D$ of $X$. For each $x\in D$ pick an
$\varepsilon'$-chain from $x$ to itself and write $C_x$ for the sequence obtained by omitting
$x$ at the end of the $\varepsilon'$-chain. Set $E = \bigsqcup_{x\in D} C_x$, and let $S:E\to E$ be the bijection
which cyclically permutes the points of each $C_x$ according to the sequential order. 
Then the map $\zeta : E\to X$ defined by taking the inclusion on each $C_x$ has $\varepsilon$-dense image, and
by uniform continuity it will satisfy $(\zeta (S^n z),T^n \zeta (z)) \in\varepsilon$ 
for every $n\in F$ if $\varepsilon'$ is taken fine enough. Thus $(X,T)$ is residually finite.

Conversely, suppose that $(X,T)$ is residually finite. Let $x\in X$ and let $\varepsilon$ be a neighbourhood 
of the diagonal in $X\times X$. By residual finiteness there is a finite set $E$, a bijection $S:E\to E$,
a map $\zeta : E\to X$ such that $(\zeta (Sz),T\zeta (z)) \in\varepsilon$ for all $z\in E$, and 
a $z_0 \in E$ such that $(\zeta (z_0 ),x)\in\varepsilon$. By a simple perturbation argument 
involving a finer choice of $\varepsilon$, we may assume that $\zeta (z_0 ) = x$. Then by applying $\zeta$
to the $S$-cycle in which $z_0$ lies we obtain an $\varepsilon$-chain for $x$. Hence $(X,T)$ is chain recurrent.
\end{proof}

In Lemma~2 of \cite{Pim83} (which is also valid in the nonmetrizable case) Pimsner
showed that a $\Zb$-system $(X,T)$ is chain recurrent if and only if there does not 
exist an open set $U\subseteq X$ which is compressed by $T$ in the sense
that $U\setminus T(\overline{U} )\neq \emptyset$ and $T(\overline U) \subseteq U$. 
From such an open set one can construct a nonunitary isometry in the crossed product 
$C(X)\rtimes_\red \Zb$ \cite[Prop.\ 8]{Pim83}, which establishes the implication
(3)$\Rightarrow$(1) in the following theorem of Pimsner from \cite{Pim83}.

\begin{theorem}\label{T-Pimsner}
For a $\Zb$-system $(X,T)$ the following are equivalent:
\begin{enumerate}
\item $(X,T)$ is pseudo-nonwandering (i.e., residually finite),

\item $C(X)\rtimes_\red \Zb$ is quasidiagonal,

\item $C(X)\rtimes_\red \Zb$ is stably finite.
\end{enumerate}
Moreover, if $X$ is metrizable then the following condition can be added to the list:
\begin{enumerate}
\item[(4)] $C(X)\rtimes_\red \Zb$ can be embedded into an AF algebra.
\end{enumerate}
\end{theorem}

Actually only the metrizable case of the above theorem was treated in \cite{Pim83}, but since
residual finiteness passes to factors and quasidiagonality is a local property one can easily derive 
from \cite{Pim83} the equivalence of conditions (1), (2), and (3) for general $\Zb$-actions
by passing through (4) as applied to metrizable factors. 
The most involved implication in the theorem is (1)$\Rightarrow$(4), while (4)$\Rightarrow$(2) follows from the fact 
that AF algebras are quasidiagonal and quasidiagonality passes to $C^*$-subalgebras, 
and (2)$\Rightarrow$(3) is true for any unital $C^*$-algebra (see Section~V.4 of \cite{Bla06}).
Note that, since crossed products by actions of $\Zb$ are always nuclear, 
for metrizable $X$ the crossed product $C(X)\rtimes_\red \Zb$ is quasidiagonal 
if and only if it is an MF algebra \cite[Thm.\ 5.2.2]{BlaKir97}.

The property of quasidiagonality can be strengthened in certain natural ways, and in view of the above 
theorem we may ask whether residual finiteness
can also be used to dynamically characterize the situation in which the 
crossed product satisfies such a stronger property. Recall for example that a $C^*$-algebra is said to be 
{\it strongly quasidiagonal} if all of its representations are quasidiagonal.
By Voiculescu's theorem \cite{Voi76}, every simple separable quasidiagonal $C^*$-algebra is strongly quasidiagonal.
Hadwin showed in \cite[Thm. 25]{Had87} that, for a compact metric space $(X,d)$ and a homeomorphism $T:X\to X$, 
the crossed product $C(X)\rtimes_\red \Zb$ is strongly quasidiagonal if and only if for every $x\in X$ there are integers 
$m,n\geq N$ such that $d(T^n x,T^{-m} x) < \varepsilon$. It is easy to verify that 
the latter condition is equivalent to {\it hereditary residual finiteness}, by which we mean that every 
subsystem of $(X,T)$ is residually finite.
We also remark that $C(X)\rtimes_\red \Zb$ is residually finite dimensional (i.e., has a
separating family of finite-dimensional representations) if and only if the periodic
points are dense in $X$ \cite[Thm.\ 4.6]{Tom92}.

The proof of Theorem 25 in \cite{Had87} can also be used
to characterize when the crossed product of a metrizable $\Zb$-system is strong NF in terms of residual finiteness.
Following some preliminary observations we will give this characterization in Theorem~\ref{T-snf}. 
Recall that a separable $C^*$-algebra $A$ is said to be 
an {\it NF algebra} if it can be expressed as the 
inductive limit of a generalized inductive
system with contractive completely positive connecting maps, and a
{\it strong NF algebra} if the connecting maps can be chosen to be complete 
order embeddings \cite{BlaKir97}. Theorem~5.2.2 of \cite{BlaKir97} gives various
characterizations of NF algebras; in particular, the following are
equivalent: (i) $A$ is an NF algebra, (2) $A$ is a nuclear MF algebra, and
(iii) $A$ is nuclear and quasidiagonal. In the strong NF case we have
by \cite{BlaKir97,BlaKir01,BlaKir00} the equivalence of the following conditions:
\begin{enumerate}
\item $A$ is a strong NF algebra, 

\item for every $\Omega\in\Fin A$ and $\varepsilon > 0$ there are a
finite-dimensional $C^*$-algebra $B$ and a complete order embedding 
$\varphi : B\to A$ such that for each $x\in\Omega$ there is a $b\in B$ 
with $\| x - \varphi (b) \| < \varepsilon$,

\item $A$ is nuclear and inner quasidiagonal,

\item $A$ is nuclear and has a separating family of irreducible quasidiagonal
representations.
\end{enumerate}

For a $\lambda > 1$, we say that a $C^*$-algebra $A$ is an {\it $\OL_{\infty ,\lambda}$ space}
if for every finite set $\Omega\subseteq A$ and $\varepsilon > 0$ there is a 
finite-dimensional $C^*$-algebra $B$ and an injective linear map $\varphi : B\to A$
with $\Omega\subseteq B$ such that $\| \varphi \|_\cb \| \varphi^{-1} : \varphi (B)\to B \|_\cb < \lambda$.
We write $\OL_\infty (A)$ for the infimum over all $\lambda > 1$ such that $A$ 
is an $\OL_{\infty ,\lambda}$ space. These notions were introduced in \cite{JunOzaRua03}
as a means to analyze and quantify the relationships between properties like nuclearity, quasidiagonality, 
inner quasidiagonality, and stable finiteness using local operator space structure.
If $A$ is a strong NF algebra then $\OL_\infty (A) = 1$ by
a straightforward perturbation argument. Whether the converse is true is an open question, 
although in \cite{Eck00} it was shown to hold under the assumption that $A$ has a finite separating family
of primitive ideals. As part of Theorem~\ref{T-snf} we will verify that the converse also holds in our situation,
so that we obtain a dynamical characterization of both strong NFness and the $\OL_\infty$ invariant equalling $1$
for crossed products of metrizable $\Zb$-systems.

\begin{lemma}\label{L-trans}
Let $(X,T)$ be a transitive metrizable $\Zb$-system and let $d$ be a compatible metric on $X$.
Let $x$ be a point in $X$ such that $\overline{\orb (x)} = X$. Then the following are
equivalent:
\begin{enumerate}
\item $(X,T)$ is residually finite,

\item $\overline{\{ T^k x \}}_{k\geq 0} \cap 
\overline{\{ T^k x \}}_{k<0} \neq\emptyset$,

\item for every $\varepsilon > 0$ and $N\in\Nb$ there are 
integers $m,n\geq N$ such that $d(T^n x,T^{-m} x) < \varepsilon$,

\item for every $\varepsilon > 0$ there are 
$m,n\in\Nb$ such that $d(T^n x,T^{-m} x) < \varepsilon$.
\end{enumerate}
\end{lemma}

\begin{proof}
(1)$\Rightarrow$(2). If (2) does not hold then 
$\overline{\{ T^k x \}}_{k\geq 0}$ equals 
$X\setminus\overline{\{ T^k x \}}_{k<0}$
and hence is a clopen set which is sent to a proper clopen subset of itself under
$T$, so that $(X,T)$ fails to be residually finite by Pimsner's characterization of chain recurrence in 
Lemma~2 of \cite{Pim83}.

(2)$\Rightarrow$(3). Let $\varepsilon > 0$ and $N\in\Nb$. Take
a $\delta > 0$ with $\delta\leq\varepsilon$ such that if 
$d(y,z)<\delta$ then $d(T^k y,T^k z)<\varepsilon$ for every $k=-N,-N+1,\dots , N$. 
Whether $x$ is periodic or nonperiodic, it is clear 
from (2) that there exist integers 
$m,n\geq 0$ such that $m+n\geq 2N$ and $d(T^n x,T^{-m} x)<\delta$. 
Set $r$ equal to $N-n$, $m-N$, or $0$ according to whether $n<N$, $m<N$, or
$m,n\geq N$. Then $m-r,n+r\geq N$ and 
$d(T^{n+r} x, T^{-m+r} x)<\varepsilon$, yielding (3).

(3)$\Rightarrow$(4). Trivial.

(4)$\Rightarrow$(1). Given $m,n\in\Nb$, the sequence
$\{ x, Tx ,\dots ,T^{n-1} x, T^{-m} x , T^{-m+1} x, \dots , T^{-1} x ,x \}$
forms an $\varepsilon$-chain precisely when $d(T^n x,T^{-m} x) < \varepsilon$, 
and so $x$ is chain recurrent. Since the chain recurrent set is closed and 
$\Zb$-invariant, we obtain (1).
\end{proof}

The following lemma is a local version of Theorem~25 of \cite{Had87} and essentially follows by 
the same argument using Berg's technique as in the proof of Lemma~3.6 in \cite{PimVoi80}.
We reproduce this argument below.
For background on induced representations see \cite{Wil07}.

\begin{lemma}\label{L-trans qd}
Let $(X,T)$ be a transitive metrizable $\Zb$-system, and let $x$ be a 
point in $X$ such that $\overline{\orb (x)} = X$. Then the following are
equivalent: 
\begin{enumerate}
\item $(X,T)$ is residually finite,

\item every irreducible representation of $C(X)\rtimes_\red \Zb$
induced from the isotropy group of $x$ is quasidiagonal,

\item $C(X)\rtimes_\red \Zb$ has a faithful irreducible
quasidiagonal representation.
\end{enumerate}
\end{lemma}

\begin{proof}
If $\orb (x)$ has finite cardinality $n$ then $T$ is a cyclic permutation. As is well known,
the crossed product $C(X)\rtimes_\red \Zb$ in this case is $^*$-isomorphic to 
$M_n \otimes C(\Tb )$, whose irreducible representations are all evidently
quasidiagonal and, up to unitary equivalence, induced from the isotropy group 
of $x$. We may thus assume that $\orb (x)$ is infinite. Then the isotropy 
group of $x$ is trivial and there is only one induced representation,
namely the irreducible representation 
$\pi : C(X)\rtimes_\red \Zb \to\cB (\ell^2 (\Zb ))$ defined by 
$\pi (f)\xi_n = f(T^n x) \xi_n$ and $\pi (u)\xi_n = \xi_{n+1}$ for all 
$f\in C(X)$ and $n\in\Zb$, where $u$ is the canonical unitary
associated to $T$ and $\{ \xi_n \}_{n\in\Zb}$ is the canonical orthonormal
basis of $\ell^2 (\Zb )$. Since $\Zb$ acts freely 
on $\orb (x)$ and subrepresentations of $\pi |_{C(X)}$ correspond to subsets of $\orb (x)$,
it follows that $\pi |_{C(X)}$ is $G$-almost free in the sense of Definition~1.12 of
\cite{ZM68}, so that $\pi$ is faithful by Corollary~4.19 of \cite{ZM68}.
We thus have (2)$\Rightarrow$(3). Note also that condition (3) implies that
$C(X)\rtimes_\red \Zb$ is a quasidiagonal $C^*$-algebra and thus we
get (3)$\Rightarrow$(1) in view of Theorem~\ref{T-Pimsner}. Let us 
then assume (1) and show that $\pi$ is quasidiagonal in order to obtain (2).

Let $\Omega$ be a finite subset of $C(X)$ and let
$n\in\Nb$. In view of the definition of a quasidiagonal representation,
we need only produce an orthogonal projection $p$ in $\cB (\ell^2 (\Zb ))$
which dominates the orthogonal projection onto the subspace spanned
by $\{ \xi_j : -n\leq j\leq n \}$ and satisfies $\| [p,\pi (f)] \| < 2/n$
for all $f\in\Omega$ and $\| [p,\pi (u)] \| < 8/n$. 
Fix a compatible metric $d$ on $X$.
By uniform continuity there exists a $\delta > 0$ 
such that if $x_1 , x_2 \in X$ and $d(x_1 , x_2 ) < \delta$ then 
$|f(T^k x_1 ) - f(T^k x_2 )| < 1/n$ for each $f\in\Omega$
and $k=0, \dots , n-1$. 
By Lemma~\ref{L-trans}
we can find $r > n$ and $s < -2n$ such that $d(T^r x, T^s x) < \delta$, so
that $|f(T^{r+k} x) - f(T^{s+k} x)| < 1/n$ for each $f\in\Omega$ and
$k=0, \dots , n-1$. For each
$f\in\Omega$ we define a perturbation $a_f \in \cB (\ell^2 (W))$ of 
$\pi (f)$ by 
$$ a_f \xi_n = \left\{ \begin{array}{l@{\hspace*{5mm}}l} 
f(T^{s+k} x)\xi_n & 
\text{if } n = r+k \text{ for some } k=0, \dots , n-1, \\ 
f(T^n x)\xi_n & \text{otherwise},
\end{array} \right. $$
in which case we have $\| a_f - \pi (f) \| < 1/n$.
 
Next we apply Berg's technique (see Section~VI.4 of \cite{Dav96}) to produce 
orthogonal
unit vectors $\zeta_k , \eta_k \in\spn (\xi_{r+k} , \xi_{s+k} )$ for each
$k=0, \dots , n-1$ and a unitary $v\in\cB (\ell^2 (\Zb ))$ such that
\begin{enumerate}
\item $v\zeta_k = \zeta_{k+1}$ and $v\eta_k = \eta_{k+1}$ for 
$k=0, \dots , n-2$,

\item $v\zeta_{n-1} = \xi_{s+n}$ and $v\eta_{n-1} = \xi_{r+n}$,

\item $v$ agrees with $\pi (u)$ on the orthogonal complement of the span of the vectors
$\xi_r , \xi_{r+1} , \dots , \linebreak[1]\xi_{r+n-1} , \linebreak[1]
\xi_s , \linebreak[1]\xi_{s+1} , \dots , \xi_{s+n-1}$, and 

\item $\| \pi (u)-v \| < 4/n$.
\end{enumerate}
Let $p$ be the orthogonal projection onto the span of the vectors
$\xi_{r+n} , \xi_{r+n+1} , \dots , \xi_{s-1} , \linebreak[1]\eta_0 , 
\linebreak[1]\eta_1 , \dots ,
\eta_{n-1}$. Since the unitary $v$ cyclically permutes these vectors we have
$[p,v] = 0$ and hence
$$ \| [p,\pi (u) ] \| \leq \| p(\pi (u)-v) \| + 
\| (v-\pi (u))p \| \leq 2 \| \pi (u)-v \| < 8/n . $$
Also, if $f\in\Omega$ then $\spn (\xi_{r+k} , \xi_{s+k} )$ is an
eigenspace for $a_f$ for every $k=0, \dots ,n-1$, so that 
$[p,a_f ] = 0$ and hence
$$ \| [p,\pi (f) ] \| \leq \| p(\pi (f) - a_f ) \| + 
\| (a_f -\pi (f))p \| \leq 2 \| a_f - f \| < 2/n , $$ 
completing the proof.
\end{proof}

\begin{theorem}\label{T-snf}
Let $(X,T)$ be a metrizable $\Zb$-system and let $d$ be a compatible metric on $X$. 
Then the following are equivalent:
\begin{enumerate}
\item there is a collection $\{ (X_i ,T) \}_{i\in I}$ of transitive residually finite
subsystems of $(X,T)$ such that $\bigcup_{i\in I} X_i$ is dense in $X$,

\item there is a dense set $D\subseteq X$ such that for every $x\in D$,
$\varepsilon > 0$, and $N\in\Nb$ there are integers $m,n\geq N$ 
for which $d(T^n x,T^{-m} x) < \varepsilon$,

\item $C(X)\rtimes_\red \Zb$ is strong NF,

\item $\OL_\infty (C(X)\rtimes_\red \Zb ) = 1$.
\end{enumerate}
\end{theorem}

\begin{proof}
(1)$\Leftrightarrow$(2). This is a simple consequence of Lemma~\ref{L-trans}.

(2)$\Rightarrow$(3). Since the set $W=\bigcup_{i\in I} X_i$ is dense in $X$ 
the canonical quotient maps $C(X)\rtimes_\red \Zb \to C(X_i )\rtimes_\red \Zb$ 
for $i\in I$ form a separating family. This can be seen by taking the
faithful representation $\pi : C(X)\rtimes_\red \Zb \to \cB (\ell^2 (W) \otimes\ell^2 (\Zb ))$
canonically induced from the multiplication 
representation of $C(X)$ on $\ell^2 (W)$ and observing that for each $i\in I$ 
the cut-down of $\pi (C(X)\rtimes_\red \Zb )$ by the orthogonal projection 
onto $\ell^2 (X_i )\otimes\ell^2 (\Zb )$ is $^*$-isomorphic to 
$C(X_i )\rtimes_\red \Zb$. Since for each $i\in I$ the crossed 
product $C(X_i )\rtimes_\red \Zb$ has a faithful irreducible 
quasidiagonal representation by Lemma~\ref{L-trans qd}, we thus conclude 
using Theorem~4.5 of \cite{BlaKir01} that $C(X)\rtimes_\red \Zb$ is strong NF.

(3)$\Rightarrow$(4). This implication holds for general $C^*$-algebras,
as mentioned prior to the theorem statement. 

(4)$\Rightarrow$(1). By Theorem~5.4 of \cite{Eck00} $C(X)\rtimes_\red \Zb$
has a separating family $\Pi$ of irreducible representations whose images are stably finite $C^*$-algebras.
By \cite{Goo72} every primitive ideal of $C(X)\rtimes_\red \Zb$ is the 
kernel of an irreducible representation induced from the isotropy group of 
some point $x\in X$. Now any two faithful irreducible representations of a 
separable prime (equivalently, separable primitive) $C^*$-algebra are 
approximately unitarily equivalent. In the antiliminal case this follows from 
Voiculescu's theorem \cite{Voi76} as every faithful irreducible representation will be essential, 
while in the non-antiliminal case
the $C^*$-algebra has an essential ideal $^*$-isomorphic to the 
compact operators and hence has only one faithful irreducible representation 
up to unitary equivalence. Consequently 
we may assume that each $\pi\in\Pi$ is induced from the isotropy group of 
some $x_\pi \in X$. If this point $x_\pi$ has nontrivial isotropy group then 
it is periodic and the system $(\overline{\orb (x_\pi )} , T)$ is 
trivially residually finite. If on the other hand $x_\pi$ has trivial isotropy group then, 
as indicated in the proof of Lemma~\ref{L-trans qd}, there is a unique induced 
representation of $C(\overline{\orb (x_\pi )})\rtimes_\red \Zb$ and it is faithful, implying that 
$C(\overline{\orb (x_\pi )})\rtimes_\red \Zb$
is stably finite and hence by Theorem~\ref{T-Pimsner} 
that $(\overline{\orb (x_\pi )} , T)$ is residually finite. It remains to observe that 
$\bigcup_{\pi\in\Pi} \overline{\orb (x_\pi )}$ is dense in $X$, which results
from the fact that $\pi (f) = 0$ for all $\pi\in\Pi$ and $f\in C(X)$ whose
support is contained in the complement of 
$\bigcup_{\pi\in\Pi} \overline{\orb (x_\pi )}$.
\end{proof}

\begin{example}
Using Theorem~\ref{T-snf} we can produce by dynamical means many 
examples of NF algebras which are not strong NF 
(cf.\ Examples~5.6 and 5.19 of \cite{BlaKir01}).
For instance, take two copies of translation on $\Zb$ each compactified with
two fixed points $\pm \infty$ and identify
$+\infty$ from each copy with $-\infty$ of the other copy. This system is
residually finite, but the transitive subsystems generated by each copy of $\Zb$ fail to
be residually finite, so that $C(X)\rtimes_\red \Zb$ is NF but not strong NF.
This example is a dynamical analogue of Example~3.2 in \cite{Eck00}.
\end{example}

\begin{example}\label{E-snf sqd}
We can also use Theorem~\ref{T-snf} to exhibit 
strong NF crossed products which are not strongly quasidiagonal. 
For instance, take translation on $\Zb$ and compactify it so that it spirals 
around to the example from the previous paragraph in both the forward and 
backward directions, spending longer and longer intervals near each of the 
two fixed points in each successive approach. The crossed product is
strong NF since the system is transitive and residually finite, but it is not strongly quasidiagonal
by the discussion in the second paragraph following Theorem~\ref{T-Pimsner}. In the 
backward direction we could instead have convergence to one of the fixed 
points, in which case we have the additional feature that the backward and 
forward limit sets of the unique dense orbit do not coincide. 
\end{example}

\end{document}